\documentclass[reqno,12pt]{amsart}
\usepackage[hyperfootnotes=false,colorlinks=true,linkcolor=blue,%
citecolor=purple,filecolor=magenta,urlcolor=cyan]{hyperref}
\usepackage{nameref,zref-xr}                    

\usepackage{amsmath,amsthm,amssymb,amsfonts,amscd}
\usepackage{epsfig}
\usepackage[all]{xy}
\usepackage{tabularx}  
\usepackage{booktabs}  
\usepackage{float}
\usepackage{mathtext}
\usepackage[T2A]{fontenc}
\usepackage{amssymb}

\usepackage{multicol}

\usepackage[all]{xy} 
\usepackage{quiver}

\usepackage[thinlines]{easytable}

\usepackage{graphicx,xcolor}
\graphicspath{{Pictures/}}
\DeclareGraphicsExtensions{.pdf,.png,.jpg}

\numberwithin{equation}{section}

\theoremstyle{plain}
\newtheorem{theorem}{Theorem}

\newtheorem{lemma}[theorem]{Lemma}

\newtheorem{proposition}[theorem]{Proposition}
\newtheorem*{theorem*}{Theorem}
\newtheorem*{proposition*}{Proposition}
\newtheorem*{conjecture*}{Conjecture}
\newtheorem*{example*}{Example}

\def\p{\partial }

\theoremstyle{definition}
\newtheorem{remark}[theorem]{Remark}

\newcommand{\CC}{{\mathbb{C}}}

\newcommand{\ZZ}{{\mathbb{Z}}}

\newcommand{\epi}{{\bf e}}

\newcommand{\Jac}{\mathrm{Jac}}
\newcommand{\Fix}{\mathrm{Fix}}

\newcommand{\id}{\mathrm{id}}

\newcommand{\bs}{{\bf s}}
\newcommand{\bv}{{\bf v}}
\newcommand{\bx}{{\bf x}}
\newcommand{\by}{{\bf y}}

\newcommand{\hess}{\mathrm{hess}}
\newcommand{\age}{\mathrm{age}}

\def\H{{\mathcal H}}

\def\O{{\mathcal O}}

\def\T{{\mathcal T}}


\numberwithin{equation}{section}

\begin{document}

\title{
Orbifold Saito theory of A and D type singularities
}

\author{Alexey Basalaev and Anton Rarovskii}
\address{
A.B.: Faculty of Mathematics, HSE University, Usacheva str., 6, 119048 Moscow, Russian Federation,
\newline
A.R.: Faculty of Mathematics, HSE University, Usacheva str., 6, 119048 Moscow, Russian Federation 
and 
Center for Advanced Studies, Skoltech, Bolshoy Boulevard 30 bld. 1, 121205 Moscow, Russian Federation
}
\email{abasalaev@hse.ru, aararovskiy@edu.hse.ru}

\date{\today}

 \maketitle

 \begin{flushright}
{\it To the memory of Wolfgang Ebeling}
\end{flushright}

 \begin{abstract}
 Saito theory associates to an isolated singularity rich structure that plays an important role in mirror symmetry. In this note we construct Saito theory for A and D type Landau--Ginzburg orbifolds. Namely, for the pairs $(f,G)$, where $f$ defines an isolated singularity of A and D type and $G$ is a group of symmetries of $f$. In total we consider five families of such pairs.
 In particular, we construct the orbifold versions of Brieskorn lattice and the Gauss--Manin connection computing them explicitly for the A and D type Landau--Ginzburg orbifolds.
\end{abstract}


\section{Introduction}
Let $f \in \CC[x_1, \dots, x_N]$ be a polynomial, defining an isolated singularity, i.e. the system of equations $\partial_{x_1}f = \partial_{x_2}f = \dots = \partial_{x_N}f = 0$ has only isolated solutions in $\CC^N$. 
One of the important invariants of the singularity $f$ is a \emph{Jacobian algebra} $\Jac(f) := \O_{\CC^N,0}/(\partial_{x_1}f, \partial_{x_2}f, \dots, \partial_{x_N}f)$. It is a finite-dimensional Frobenius algebra and we denote $\mu = \text{dim}_\CC\Jac(f)$. 

Fixing a $\CC$--basis $\{\phi_{\alpha}(\bx)\}^{\mu-1}_{\alpha = 0}$ of $\Jac(f)$, associate to $f$ the \emph{universal unfolding}
\[
    F(\bx,\textbf{s}) = f (\bx)+ \sum^{\mu-1}_{\alpha=0} s_{\alpha} \phi_\alpha(\bx),
\]
where $s_0,\dots,s_{\mu-1}$ are the parameters varying in the \emph{base space} $\mathcal{S}$ --- open neighborhood of $0 \in \CC^\mu$. 

Let $\H_f^{(0)}$ and $\H_F^{(0)}$ denote the \emph{Brieskorn lattices} of $f$ and $F$
\begin{align*}
    \H_f^{(0)} & := H^*(\Omega^{\bullet}_{\CC^N}\otimes \CC[[z]], z d + df\wedge),
    \\ 
    \H_F^{(0)} & := H^*(\Omega^{\bullet}_{\CC^N}\otimes \CC[[z]] \otimes \O_{\mathcal{S}}, z d + dF\wedge).
\end{align*}
These are classical objects developed by \cite{SK81,SK83, SM89}. 
The cohomology groups above are concentrated at the top grading $N$ and $\H_f^{(0)}$, $\H_F^{(0)}$ are rank $\mu$ free $\CC[[z]]$-modules.
Essentially, $\H_F^{(0)}$ can be viewed as a deformation of $\H_f^{(0)}$ over the base $\mathcal{S}$.

K.~Saito associated the connection to the completion of $\H_F$. Let $\CC((z))$ stand for the Laurent series in $z$. Define
\[
    \H_f := \H_f^{(0)} \otimes_{\CC[[z]]} \CC((z)), \quad \H_F := \H_F^{(0)} \otimes_{\CC[[z]]} \CC((z)).
\]
With the help of the unfolding $F$, K.~Saito introduces the Gauss--Manin connection 
\begin{align*}
 & \nabla^F:  \mathcal{TS} \otimes \H_F \to  \H_F,
 \ \nabla^F_X\zeta := \left[ (X \cdot \phi + z^{-1} \phi (X \cdot F)) dx_1 \wedge \dots \wedge dx_N \right].
\end{align*}
where $\zeta = [\phi dx_1 \wedge \dots \wedge dx_N] \in \H_F$ .

Endowed with the higher residue pairing and the primitive form, the pair $(\H_F, \nabla^F)$ gives rise to the structure of Dubrovin--Frobenius manifold \cite{SaiT},\cite{SK83}.
Such Dubrovin--Frobenius manifolds played the key role in mirror symmetry conjecture (cf. \cite{CR11}).

The main goal of our work is to construct an orbifold version of $(\H_F, \nabla^F)$ and investigate its properties.

\subsection*{Orbifold theories} Set $G_f := \lbrace \alpha \in (\CC^\ast)^N \ | \ f(\alpha \cdot \bx) = f(\bx) \rbrace$. Let $G \subseteq G_f$.
Then the pair $(f,G)$ is called \emph{Landau-Ginzburg orbifold}. Investigation of these objects was initiated by physicists in early 90s (cf. \cite{IV}, \cite{V}, \cite{Witt}), and was followed by mathematicians. In \cite{BTW1}, \cite{BT1} the authors introduced the algebra $\Jac(f,G)$, which is $G$-equivariant version of Jacobian algebra $\Jac(f)$. In particular, as a vector space, $\Jac(f,G) \cong \oplus_{g \in G} B_g$ for some vector spaces $B_g$. We have $B_\id \cong (\Jac(f))^G$ while the other direct summands are more complicated. In particular, some of $B_g$ can be zero vector spaces.
However the construction of the orbifold versions of $\H_F$ and $\nabla^F$ appeared to be much more involved.

Some approaches towards the construction of an orbifold Saito theory appeared in the literature. Junwu Tu introduced in \cite{Tu21a,Tu21b} the categorical approach to equivariant Saito theory using category of $G$-equivariant matrix factorizations and the variations of the semi-infinite Hodge structures (see also \cite{AT22}). W.~He, S.~Li and Y.~Li considered in \cite{HLL} the deformation of $A_\infty$--structure of two orbifold examples.
In our work we are going to use a more classical approach and introduce the orbifold Saito theory for A and D type singularities via the unfoldings $F$.

Let $\mathcal{S}_{(f,G)}$ stand for the small open neighbourhood of the origin in the $\dim \Jac(f,G)$--dimensional complex vector space. 
Denote the orbifold versions of $\H_f$ and $\H_F$ by $\H_{(f,G)}$ and $\H^{def}_{(f,G)}$. Define them via the following isomorphism
\[
    \H_{(f,G)} \cong \CC((z)) \otimes \Jac(f,G), \qquad \H^{def}_{(f,G)} \cong \CC((z)) \otimes \Jac(f,G) \otimes \O_{\mathcal{S}_{(f,G)}}.
\]
In what follows for any $g \in G$ we will denote by $[\phi \xi_g]$ the elements of $B_g$ and by $\Pi_g$ the projection
\[
    \Pi_g : \H^{def}_{(f,G)} \to \CC((z)) \otimes B_g \otimes \O_{\mathcal{S}_{(f,G)}} \subset \H^{def}_{(f,G)}
\]
acting trivially on the first and last $\otimes$--multiples.

Denote the coordinates of $\mathcal{S}_{(f,G)}$ by $v_{g}^\alpha$ with $\alpha$ running up to $\dim B_g$ for all $g\in G$. 
Assuming every $g$--indexed coordinate $v_g^\alpha$ to have $G$--grading equal to $g^{-1}$, the space $\mathcal{S}_{(f,G)}$ becomes $G$--graded. Then $\T\mathcal{S}_{(f,G)}$ and $\O_{\mathcal{S}_{(f,G)}}$ inherit the $G$--grading as well.
Define the $G$--grading of $z$ to be $\id$. Then the spaces $\H_{(f,G)}$ and $\H^{def}_{(f,G)}$ are $G$--graded, distributing the $G$--grading over the tensor product multiplicatively.


We employ orbifold equivalence in order to introduce the orbifold version of the Gauss--Manin connection 
\[
    \nabla^{(f,G)}: T\mathcal{S}_{(f,G)} \otimes \H^{def}_{(f,G)} \to \H^{def}_{(f,G)}.
\]

\subsection*{Orbifold equivalence}
Let $\tau \colon \widehat{\CC^N/G} \to \CC^N/G$ be a \emph{crepant resolution} of the quotient singularity $\CC^N/G$. In particular, we have the relation of the canonical classes $\tau^*K_{\CC^N/G} = K_{\widehat{\CC^N/G}}$. 

Let $\widehat{\CC^N/G}$ be covered by the smooth charts $U_1, \dots, U_k$. 
We say that $(f,G)$ is \textit{orbifold equivalent} to $\overline{f}$, if $\overline{f}$ is the unique lift of $f$ to the smooth charts $U_i$ that has singularities different from Morse singularities. Adopt also the notation $(f,G)\sim (\overline{f},\{\id\})$.
Existence of a crepant resolution does not guarantee existence of such $\overline{f}$. In particular, it can happen that all the lifts of $f$ have non-Morse singularities.

The notion of orbifold equivalence can also be introduced in the context of homological algebra. (cf. \cite{Io}, \cite{Orl1}, \cite{Orl2} and Section~6 of \cite{BT1}).

In \cite{BTW1} authors established the orbifold equivalence for all polynomials ${f \in \CC[x_1,x_2,x_3]}$ defining ADE singularities and $G \cong \ZZ/2\ZZ$ (see Table~\ref{table1}). Moreover, the authors found the explicit isomorphism $\Psi_{\bar{f}}$ between the Frobenius algebras $\Jac(\bar{f})$ and $\Jac(f,G)$.

In the current work we use the classical Saito theory of $\overline{f}$ and this isomorphism $\Psi_{\overline f}$ in order to give the definition of the orbifold Saito theory of $(f,G)$.

\begin{table}[h]
{\small
\begin{center}
\begin{tabular}{l||l|c||c}
&$f(x_1,x_2,x_3)$ & $G$ & $\overline{f}(y_1,y_2,y_3)$\\
\hline
1.&$x_1^{k+1}+x_2^2+x_3^2$,\quad $k \ge  1$ & $\left<\frac{1}{2}(0,1,1)\right>$ & $y_1^{k+1}+y_2+y_2y_3^2$\\
2.&$x_1^{2k}+x_2^2+x_3^2$,\quad $k \ge  1$ & $\left<\frac{1}{2}(1,0,1)\right>$ & $y_1^2+y_2^k+y_2y_3^2$\\
3.&$x_1^2+x_2^2+x_2x_3^{2k}$,\quad $k \ge  1$ & $\left<\frac{1}{2}(1,0,1)\right>$ & $y_1^2+y_1y_2^k+y_2y_3^2$\\
4.&$x_1^2+x_2^{k-1}+x_2x_3^2$,\quad $k \ge  4$ & $\left<\frac{1}{2}(1,0,1)\right>$ & $y_1^{k-1}+y_1y_2+y_2y_3^2$\\
5.&$x_1^2+x_2^2+x_2x_3^{2k+1}$,\quad $k \ge  1$ & $\left<\frac{1}{2}(0,1,1)\right>$ & $y_1^2+y_3y_2^2+y_2y_3^{k+1}$\\
\hline
\end{tabular}
\end{center}
}
\smallskip
\caption{$(f,G)\sim (\overline{f},\{\id\})$}\label{table1}
\end{table}



\subsection*{Main results} 
Consider the pair $(f,G)$ and a miniversal unfolding $F = f + \sum_{\alpha=0}^{\mu-1} s_{\alpha}\phi_\alpha$.
Denote by $F^G$ the $G$-invariant part of $F$. Namely, $F^G := f + \sum_{\beta \in R} s_{\beta}\phi_\beta$ where $R$ is the set of all indices $\beta$, such that $\phi_\beta$ is $G$--invariant. Such variables $s_\beta$ can be identified with the $\id$--graded coordinates $v^\bullet_{\id}$ of $\mathcal{S}_{(f,G)}$. So, we can consider $F^G = F^G(\bx,v_\id^\bullet)$.

For any $g \in G$ set $I_g := \lbrace k \ | \ g \cdot x_k = x_k \rbrace $ and $I_g^c := \lbrace 1,2,3 \rbrace \backslash I_g$. Consider the polynomial
\[
    H_g^F := \widetilde m \det \left( \frac{\p^2 F^G}{\p x_i \p x_j} \right)_{i,j \in I_g^c} \in \CC[\bx,v_\id^\bullet]
\]
with $\widetilde m \in \CC^\ast$, the certain constant fixed in Section~\ref{section: H_g}.

\begin{theorem}\label{theorem: main}
    Let $(f,G)$ be a Landau-Ginzburg orbifold from Table \ref{table1}. Then
    
    \begin{enumerate}
     \item $\nabla^{(f,G)}$ is torsion-free and respects the $G$--grading.
     \item for any $[\phi \xi_\id] \in \H_{(f,G)}^{def}$ we have
     \[
        \Pi_\id \cdot \nabla_{\frac{\partial}{\partial v^\alpha_\id}}^{(f,G)}[\phi \xi_\id] 
        = 
        \nabla^{F^G}_{\frac{\partial}{\partial v^\alpha_\id}}[\phi] \xi_{\id}
    \] 
     \item for any $[\phi \xi_g] \in \H_{(f,G)}^{def}$ we have
     \[
        \Pi_g \cdot \nabla_{\frac{\partial}{\partial v^\alpha_\id}}^{(f,G)}[\phi \xi_g] 
        = 
        \nabla^{F^G}_{\frac{\partial}{\partial v^\alpha_\id}}[\phi] \xi_g
        \ \text{and} \
        \Pi_g \cdot \nabla_{\frac{\partial}{\partial v^\alpha_g}}^{(f,G)}[\phi \xi_\id] 
        = 
        \nabla^{F^G}_{\frac{\partial}{\partial v^\alpha_\id}}[\phi] \xi_g.
    \]         
     \item for any $[\phi \xi_g] \in \H_{(f,G)}^{def}$ we have
     \[
        \nabla_{\frac{\partial}{\partial v^\alpha_{g}}}^{(f,G)}[\phi \xi_{g}] = c_g \nabla_{\frac{\partial}{\partial v^\alpha_{\id}}}^{(f,G)}[\phi H_g^F] (\widetilde v) \xi_{\id}
    \]         
    where $\widetilde v = \widetilde v(v)$ is the certain linear change of the variables, and 
        \[
            c_g = \frac{1}{[G]^2} (-1)^{\frac{1}{2}(N-N_g)(N-N_g-1)} \epi[- \frac{1}{2}\age (g)] \in \CC^\ast
        \]
        as in Section~\ref{section: H_g}.
    \end{enumerate}
\end{theorem}

Part (ii) of the theorem above is very important for the Dubrovin--Frobenius manifold of an orbifold Saito theory. Such condition was assumed to hold true in \cite{BT2} in order to construct such Dubrovin--Frobenius manifold axiomatically.

Expanding the theorem above we also compute explicitly the connection $\nabla^{(f,G)}$ for Landau--Ginzburg orbifolds from Table~\ref{table1}.
Our computations show that the connection $\nabla^{(f,G)}$ agrees with the product of $\Jac(f,G)$.

The methods, that we develop in this paper can also be used in a wider context. In particular, orbifold equivalence was established by the second author for the series of non--invertible quasihomogeneous singularities (cf. \cite{R1}). Our approach allows one to define orbifold Saito theory for any such Landau--Ginzburg orbifold.

\subsection*{Acknowledgements} 
The work of Alexey Basalaev and Anton Rarovskii was supported by the Theoretical Physics and Mathematics Advancement Foundation "BASIS".

\section{Brieskorn lattice and Gauss--Manin connection}

Consider $\CC[\bx] := \mathbb{C}[x_1,x_2, \dots, x_N]$ a ring of polynomials with complex coefficients. We call a polynomial $f\in \CC[\bx]$ \emph{non-degenerate} if $f$ defines an isolated singularity at the origin, i.e. the system of equations $\{\frac{\partial f}{\partial x_1} = \frac{\partial f}{\partial x_2} = \dots = \frac{\partial f}{\partial x_N}\ = 0\}$ has the unique solution $0 \in \CC^N$. 

Associate to $f$ the \emph{Jacobian algebra} $\Jac(f) := \O_{\CC^N}/_{(\frac{\partial f}{\partial x_1}, \frac{\partial f}{\partial x_2}, \dots, \frac{\partial f}{\partial x_N})}$.
Its dimension $\mu$, called \emph{Milnor number}, is finite if and only if $f$ is non-degenerate (see \cite{AGV}).


The group of \emph{maximal diagonal symmetries} of $f$ is defined by:
\begin{align*}
    G_f := \{ g \in (\mathbb{C}^*)^N \:|\: f(g \cdot \bx ) = f(\bx) \}.
\end{align*}
A pair $(f,G)$ with $G\subseteq G_f$ is called a \emph{Landau-Ginzburg orbifold}. 

With each element $g\in G_f$ we can associate the algebra $\mathrm{Jac}(f^g)$ of the polynomial $f^g = f|_{\Fix(g)}$ where $\Fix(g) = \{\bx \in \CC^N \ | \ g\cdot \bx= \bx \}$. If $\Fix(g) = 0$, set $f^g := 1$ and $\mathrm{Jac}(f^g) := \CC$. 

\begin{proposition}[\cite{ET1}, Prop. 5]
 If the polynomial $f$ is non-degenerate, then $\forall g\in G$ such that $\Fix(g) \neq 0$ the polynomial $f^g$ is also non-degenerate.
\end{proposition}

\subsection{Orbifold local algebra}
We associate with a Landau-Ginzburg orbifold $(f,G)$ the algebra $\Jac(f,G)$ which is $G$-equivariant generalization of $\Jac(f)$. 
$\Jac(f,G)$ is defined either as Hochschild cohomology of the category of $G$-equivariant matrix factorizations or axiomatically as in \cite{BTW1,BTW2}. Two definitions agree for the polynomials we consider in this text (see \cite{Sh}, \cite{BT1} for details). 

It was proved in \cite{BTW1}, that for any Landau--Ginzburg orbifold $(f,G)$ from Table~\ref{table1} there is a non-degenerate polynomial $\overline f \in \CC[y_1,y_2,y_3]$ and an explicit map $\Psi_{\overline f}$, giving the algebra isomorphism
\begin{equation}\label{eq: orb iso}
 \Psi_{\overline f}: \Jac(\overline f) \xrightarrow{\sim} \Jac(f,G).
\end{equation}

By the construction we have $\Jac(f,G) \cong \oplus_{g \in G} B_g$ with $B_\id = (\Jac(f))^G$ and $B_g \subseteq \Jac(f^g)$. 
In order to distinguish between the direct summands of $\Jac(f,G)$ it will be useful to denote an element of $B_g$ by $[\phi]\xi_g$, where $[\phi]$ stand for the class is $\Jac(f^g)$ and $\xi_g$ is the formal letter keeping track of the direct summand $B_g$. 

We skip full definition of $\Jac(f,G)$ because in the current text $\Jac(f,G)$ can be fully reconstructed via the explicit isomorphism $\Psi_{\bar f}$. However the following will be very important in what follows.

\subsubsection{Key products}\label{section: H_g}
For a fixed pair $(f,G)$ define a polynomial 
$H_{g} \in \CC[\bx]$ by
\begin{equation}
H_{g}:=
\begin{cases}
\widetilde{m}_{g}\det \left(\frac{\partial^2 f}{\partial x_{i} \partial x_{j}}\right)_{i,j\in I_{g}^c} 
& \text{if} \quad I_{g}^c \neq \emptyset\\
1 & \text{if} \quad I_{g}^c = \emptyset
\end{cases},
\end{equation}
where $\widetilde{m}_{g}\in\CC^\ast$ is the constant uniquely determined by the following equation in $\Jac(f)$
\begin{equation}\label{H and hessians}
\frac{1}{\mu_{f^{g}}}[\hess(f^{g})H_{g}]=\frac{1}{\mu_{f}}[\hess(f)],
\end{equation}
where $\mu_{f^{g}}$ and $\hess(f^g)$ stand for the Milnor number and Hessian polynomial of $f^g$ respectively.

After the suitable choice of the basis, the product $\circ$ in $\Jac(f,G)$ satisfies
\[
    [a]\xi_g \circ [b]\xi_{g^{-1}} = c_g [a b H_g]\xi_\id
\]
where on the right hand side we have the projection to $(\Jac(f))^G$ of the product $[a][b][H_g]$ computed $\Jac(f)$ and 
\[
    c_g := \frac{1}{[G]^2} (-1)^{\frac{1}{2}(N-N_g)(N-N_g-1)} \epi[- \frac{1}{2}\age (g)] \in \CC^\ast.
\]
\begin{remark}
The formula above is the rescaling of Eq.(4.26) from \cite{BTW1}. Exactly this rescaling was used in Section~5 of loc.cit.. 
\end{remark}
\begin{remark}
The polynomials $H_g$ were introduced in a more general context in \cite{BTW1} (cf. Definition~27). In the notation of loc.cit. we only assume $H_{g,g^{-1}}$.
\end{remark}

\begin{remark}
    We only consider the so--called "diagonal" symmetry groups in this text. Nevertheless one may consider even nonabelian symmetries as well. See \cite{B25,BI22,BI24,BI21}.
\end{remark}

\subsection{Elements of Saito Theory} 
Let $\{\phi_\alpha\}_{\alpha=0}^{\mu-1}$ be a $\CC$--basis of $\Jac(f)$. Then we define the \emph{universal unfolding} of $f$ to be the following polynomial
\[
    F(\bx,\mathbf{s}) = f(\bx) + \sum_{\alpha = 0}^{\mu-1} s_\alpha\phi_\alpha(\bx) \in \CC[\bx,\mathbf{s}].
\] 
In what follows assume $s_0,\dots, s_{\mu-1}$ to vary in the \textit{base space of the unfolding} $\mathcal{S}$ --- small ball in $\CC^\mu$ centered at the origin $0 \in \CC^\mu$. 

The \emph{Brieskorn lattice} and \emph{completed Brieskorn lattice} of polynomial $f$ is a cohomology ring:
\[ 
\H^{(0)}_f := H^*(\Omega^{\bullet}_{\CC^N}\otimes \CC [[z]], z d + df\wedge) 
    \quad \H_f = \H^{(0)}_f \otimes_{\CC [[z]]} \CC((z))
\]
It's easy to see that 
\[
    \H^{(0)}_f \cong (\Omega^{N}_{\CC^N}\otimes \CC [[z]])/(z d + df).
\]

In the same way, we define the Brieskorn lattices for the universal unfolding that can be seen as a deformation of $\H^{(0)}_f$:
\[ 
    \H^{(0)}_F := H^*(\Omega^{\bullet}_{\CC^N}\otimes \O_{\mathcal S} \otimes \CC [[z]], z d + dF), 
     \quad 
     \H_F := \H^{(0)}_F \otimes_{\CC [[z]]} \CC((z)).
\]     
In what follows assume the notation $D_f := z d + df$ and $D_F := z d + dF$. 
It is well-known that 
\[
\H_{f} \cong \CC((z)) \otimes \Jac(f) \quad \text{and} \quad \H_{F} \cong \CC((z)) \otimes \Jac(f) \otimes \mathcal{O_{S}}.
\]
%


Gauss-Manin connection $\nabla^F$ is defined on the chain level by
\[
    \nabla^F:  \mathcal{TS} \otimes \Omega^{\bullet}_{\CC^N}\otimes \O_{S} \otimes \CC ((z)) \to  \Omega^{\bullet}_{\CC^N}\otimes  \O_{S} \otimes\CC ((z))
\]
\[
\nabla^F_X \left( \alpha d^m x \right) := \left(X \cdot \alpha + z^{-1} \alpha (X \cdot F)  \right) d^mx, \quad m \ge 0
\]
where $\alpha = \alpha(\bs,\bx,z)$. This connection is well--defined on $\H_F$ as well because $[\nabla^F, D_F] = 0$.

\begin{remark}
 Classical Saito theory contains also such important ingredients as higher residue pairing, good section and primitive form. We do not introduce these objects here because they do not play an important role in the current work.
\end{remark}

\section{Orbifold Brieskorn lattice}

Given a Landau--Ginzburg orbifold $(f,G)$, let $\mathcal{S}_{(f,G)}$ stand for the small open neighbourhood of the origin in the $\dim \Jac(f,G)$--dimensional complex vector space. Let $k_g := \dim B_g$. The coordinates of $\mathcal{S}_{(f,G)}$ will be denoted by $v_{g}^{\alpha}$ with ${\alpha = 0,\dots,k_g -1}$ and $g \in G$.
Denote the orbifold versions of $\H_f$ and $\H_F$ by $\H_{(f,G)}$ and $\H^{def}_{(f,G)}$ respectively. Define them via the following isomorphism
\[
    \H_{(f,G)} \cong \CC((z)) \otimes \Jac(f,G), \qquad \H^{def}_{(f,G)} \cong \CC((z)) \otimes \Jac(f,G) \otimes \O_{\mathcal{S}_{(f,G)}}.
\]
In order to define the orbifold Gauss--Manin connection we make use of the ``classical'' Gauss--Manin connection of the orbifold equivalent polynomial.

Let $(f,G) \sim (\overline{f},\{\id\})$ and $\Psi_{\bar f}$ be the isomorphism of Eq.~\eqref{eq: orb iso}.
When $G$ is of order $2$ generated by $g$, fix a $\CC$--basis $\lbrace [\phi_{\alpha, \id}]\rbrace_{\alpha=0}^{k_\id-1} \sqcup \lbrace [\phi_{\beta, g}] \rbrace_{\beta=0}^{k_g-1}$ of the Jacobian algebra $\Jac(\bar{f})$ 
such that 
\begin{equation}\label{eq: g-sector generators}
        \Psi_{\bar{f}}([\phi_{\alpha, \id}]) \in B_{\id} \ \text{ and } \ \Psi_{\bar{f}}([\phi_{\beta, g}]) \in B_{g}.
\end{equation}
Fix the universal unfolding $\bar{F}(\by, \bar{\textbf{s}}, \bar{\textbf{v}})$ of $\bar{f} \in \CC[\by]$ by
\[
    \bar{F}(\by, \bar{\textbf{s}}, \bar{\textbf{v}}) = \bar{f} + \sum_{\alpha=0}^{k_\id-1} \bar{s}_\alpha\phi_{\alpha, id} + \sum_{\beta=0}^{k_g-1} \bar{v}_\beta\phi_{\beta,g}.
\]
Denote by $\bar{S} \subseteq \CC^{\mu}$ the base space of this unfolding. 

Extend $\Psi_{\bar{f}}$ to the map $\Psi_{\overline S}: \Jac(\overline f) \otimes \O_{\mathcal{\bar{S}}} \to \Jac(f,G) \otimes \O_{\mathcal{S}_{(f,G)}}$ acting by $\Psi_{\bar{f}}$ on the first multiple and by  
\[
    \Psi_{\overline S}(\bar{s}_{\alpha}) := v_{\id}^{\alpha} \ \text{ and } \ \Psi_{\overline S}(\bar{v}_\beta) := v_g^\beta
\]
on the second multiple.

Define the \emph{orbifold Gauss-Manin connection}
\[
    \nabla^{(f,G)}: T\mathcal{S}_{(f,G)} \otimes \H^{def}_{(f,G)} \to \H^{def}_{(f,G)}
\]
by the commutativity of the following diagram:
\[ \begin{tikzcd}
\T \bar{\mathcal{S}} \otimes \H_{\bar F} \arrow{r}{\nabla^{\bar F}} \arrow[swap]{d}{(\Psi_{\bar{\mathcal{S}}})_* \otimes \Psi_{\bar{\mathcal{S}}}} & \H_{\bar F} \arrow{d}{\Psi_{\bar{\mathcal{S}}}} \\%
\T \mathcal{S}_{(f,G)} \otimes \H^{def}_{(f,G)}  \arrow{r}{\nabla^{(f,G)}}& \H^{def}_{(f,G)}
\end{tikzcd}
\]
In particular
\[
\nabla^{(f,G)}_{X}[A] = \Psi_{\bar{\mathcal{S}}} \circ \nabla^{\bar F}_{Y} \circ  \Psi^{-1}_{\bar{\mathcal{S}}}(A)
\]
where $Y := (\Psi_{\bar{\mathcal{S}}}^{-1})_*(X)$ is a push-forward of a  vector field on $\mathcal{S}_{(f,G)}$ and $A \in \H_{(f,G)}^{def}$.


\subsection{G-grading on orbifold Brieskorn lattice}\label{section: BL G--grading}
$\Jac(f,G) = \bigoplus_{g \in G} B_g$ has the natural $G$-grading. 
Let $\O_{\mathcal{S}_{(f,G)}}$ be $G$-graded by setting $v^{\alpha}_{g}$ to have $G$--grading $g^{-1}$. Then $v^{\alpha}_{g}[\phi_\beta]\xi_g \in \Jac(f,G)\otimes \O_{\mathcal{S}_{(f,G)}}$ has the identity grading. The $G$-grading is also induced on $\T \mathcal{S}_{(f,G)}$ by setting $\frac{\partial}{\partial v^{\alpha}_{g}}$ to have the grading $g$. Composing these $G$-gradings and assuming $z$ to have identity grading, $\H^{def}_{(f,G)}$ is $G$--graded as well.

Call $\mathcal{S}^{\id}_{(f,G)} := \left.\mathcal{S}_{(f,G)}\right\vert_{v^\alpha_g=0\; \forall g \neq id}$ the \emph{identity subspace} of $\mathcal{S}_{(f,G)}$
and the space 
\[
    \H_{\id} = \CC((z)) \otimes B_\id \otimes \O_{\mathcal{S}^{\id}_{(f,G)}} \subseteq \H^{def}_{(f,G)} 
\]
\emph{the strict identity sector of $\H^{def}_{(f,G)}$}.

\begin{remark}\label{remark: strict and identity sectors}
Note that in general $\H_{\id}$ is smaller than the space of all $\id$--graded elements of $\H^{def}_{(f,G)}$. The latter contains for example the element $v_g^{\alpha} [\phi_\beta] \xi_g$ that is not an element of the strict identity sector.
\end{remark}

Denote by $\pi_\id$ the projection on $\H_{\id}$ 
\[
    \pi_\id: \H^{def}_{(f,G)} \twoheadrightarrow \H_{\id}.
\]

Assume the classes $\Psi_{\bar{f}}([\phi_{\alpha, \id}])$ from Eq.~\eqref{eq: g-sector generators} to be equal to the classes $[\phi_\alpha]$ in $\Jac(f)$.
Denote
\[
    F^G(\bx,\bs) := f(\bx) + \sum_{\alpha=0}^{k_\id-1} s_\alpha \phi_\alpha(\bx).
\]
This function differs from the unfolding $F$ by the summation on the right. Namely, it consists only of those summands of $F$, that are $G$--invariant.

Then the Gauss-Manin connection of $F^G$ is well-defined and maps ${\T \mathcal{S}^{\id}_{(f,G)} \otimes \H_\id \to \H_\id}$.

Consider one more piece of notation.
Define the polynomial 
$H_{g}^F \in \CC[\bx,\bs]$ by
\begin{equation}
H_{g}^F :=
\begin{cases}
\widetilde{m}_{g}\det \left(\frac{\partial^2 F^G}{\partial x_{i} \partial x_{j}}\right)_{i,j\in I_{g}^c} 
& \text{if} \quad I_{g}^c \neq \emptyset\\
1 & \text{if} \quad I_{g}^c = \emptyset
\end{cases},
\end{equation}
where $\widetilde{m}_{g}\in\CC^\ast$ is the constant defined in Section~\ref{section: H_g}.
It follows immediately that
\[
    H_{g}^F(\bx,\bs) \mid_{\bs = 0} \ = H_g(\bx)
\]
with the latter polynomial introduced by \cite{BTW1} while computing the multiplication table of $\Jac(f,G)$.

\subsection{Proof of Theorem~\ref{theorem: main}}
The connection is torsion-free by the construction. It will be proved in Section~\ref{section: proof-1} that it respects the $G$--grading, what gives us part (i). 

Part (ii) will also be proved in Section~\ref{section: proof-1} except the last Landau--Ginzburg orbifold from Table~\ref{table1}. This case together with part (iii) of the theorem will be proved in Section~\ref{section: proof-2} via the explicit computations.

%
%

\section{Brieskorn lattice of a crepant resolution}\label{section: proof-1}

We start by recalling the isomorphism $\Psi_{\bar f}$ between $\Jac(f,G)$ and $\Jac(\bar{f)}$ for each pair from Table~\ref{table1}.
We also introduce the specific $\ZZ/2\ZZ$--action on $\H_{\bar F}$ that we denote by $p$ for each orbifold equivalent pair, except the last one.
This $\ZZ/2\ZZ$--action on $(\by, \bar{\bs}, \bar{\bv}) \in \CC^3 \times \CC^{\mu}$ will satisfy the following properties:
\[
p(\phi_{\alpha, id}) = \phi_{\alpha, id}, \quad p(\phi_{\beta, g}) = -\phi_{\beta, g}, \quad p(\bar{s}_\alpha) = \bar{s}_\alpha, \quad p(\bar{v}_\beta) = - \bar{v}_\beta.
\]

It defines $\ZZ/2\ZZ$-grading on $\Jac(\bar{f})$ and on $\O_{\bar{S}}$. It also lifts to the  $\ZZ/2\ZZ$-grading on $\H_{\bar{F}} = \H_{\bar{F}, 0} \oplus \H_{\bar{F},1}$, $\T \mathcal{\bar{S}} = \T \mathcal{\bar{S}}^{0} \oplus \T \mathcal{\bar{S}}^1$ and $\O_{\mathcal{\bar{S}}} = \O_{\mathcal{\bar{S}}}^{0} \oplus \O_{\mathcal{\bar{S}}}^{1}$, where the indices $0$ correspond to $p$-invariant subspaces. 
Denote also
\[
\H_{\bar{F},\id} := \CC((z)) \otimes \CC \langle \phi_{0,\id},\dots, \phi_{k_\id-1,\id} \rangle dy^1\wedge dy^2 \wedge dy^3  \otimes \O_{\mathcal{\bar{S}}}^{0} \subset \H_{\bar{F},0}.
\]
Similarly to the strict identity sector $\H_{\id} \subset \H^{def}_{(f,G)}$, $\H_{\bar{F},\id}$ is generally smaller than $\H_{\bar{F},0}$ (compare with Remark~\ref{remark: strict and identity sectors}).

\subsubsection*{\textbf{1.} $(A_k, \ZZ/2\ZZ) \sim (A_{k}\oplus D_2, \id)$ }
\begin{gather*}
    f :=x_1^{k+1}+x_2^2+x_3^2,\quad G:=\left<g\right>,\ g:=\frac{1}{2}(0,1,1), \quad \overline{f}:=y_1^{k+1}+y_2+y_2y_3^2,
    \\
    \bar{F} (y, \bar{s}, \bar{v}) = \bar{f} + \sum_{\alpha=0}^{k-1}\bar{s}_\alpha y_1^\alpha +  \sum_{\beta=0}^{k-1} \bar{v}_\beta y_1^\beta y_3, \quad F^G(s,z) = f + \sum_{\alpha=0}^{k-1} s_\alpha x_1^{\alpha},    
    \\
    \Psi_{\bar{f}}([y_1]) = [x_1]\xi_{\id}, \ \Psi_{\bar{f}}([y_2]) = [x_2^2]\xi_{\id}, \ \Psi_{\bar{f}}([y_3]) = \xi_{g}, \
    \Psi (\overline s_\alpha) := v_\id^{\alpha}, \ \Psi (\overline v_\alpha) := v_g^{\alpha},
    \\
    p(y_1) = y_1,\quad p(y_2) = y_2, \quad p(y_3) = -y_3, \quad p(\bar{s}_\alpha) = \bar{s}_\alpha, \quad p(\bar{v}_\alpha) = - \bar{v}_\alpha.
\end{gather*}

\subsubsection*{\textbf{2.} $(A_{2k-1}, \ZZ/2\ZZ) \sim (D_{k+1}, \id)$}
\begin{gather*}
f:=x_1^{2k}+x_2^2+x_3^2,\quad G:=\left<g\right>,\ g:=\frac{1}{2}(1,0,1), \quad \overline{f}:=y_1^2+y_2^k+y_2y_3^2,
\\
\bar{F} (y, \bar{s}, \bar{v}) = \bar{f} + \sum_{\alpha=0}^{k-1} \bar{s}_\alpha y_2^\alpha + \bar{v}_0y_3, \quad F^G(s,z) = f + \sum^{k-1}_{\alpha = 0} s_\alpha x_1^{2\alpha},
\\
\Psi_{\bar{f}}([y_1]) = [x_2]\xi_{\id}, \quad \Psi_{\bar{f}}([y_2]) = [x_1^2]\xi_{\id}, \quad \Psi_{\bar{f}}([y_3]) = \xi_{g},
\Psi (\overline s_\alpha) := v_\id^{\alpha}, \ \Psi (\overline v_0) := v_g^{0},
\\
p(y_1) = y_1, \quad p(y_2) = y_2, \quad p(y_3) = -y_3, \quad p(\bar{s}_\alpha) = \bar{s}_\alpha, \quad p(\bar{v}_\alpha) = - \bar{v}_\alpha.
\end{gather*}

\subsubsection*{\textbf{3.} $(A_{4k-1}, \ZZ/2\ZZ) \sim (D_{2k+1}, \id)$}
\begin{gather*}
f:=x_1^{2}+x_2^2+x_2x_3^{2k},\quad G:=\left<g\right>,\ g:=\frac{1}{2}(1,0,1), \quad \overline{f}:=y_1^2+y_1y_2^k+y_2y_3^2,
\\
\bar{F} (y, \bar{s}, \bar{v}) = \bar{f} + \sum^{2k-1}_{\alpha = 0} \bar{s}_\alpha y_2^\alpha  + \bar{v}_0y_3, \quad F^G(s,z) = f + \sum_{\alpha = 0}^{2k-1} s_\alpha x_3^{2\alpha},
\\
\Psi_{\bar{f}}([y_1]) = [x_2]\xi_{\id}, \quad \Psi_{\bar{f}}([y_2]) = [x_3^2]\xi_{\id}, \quad \Psi_{\bar{f}}([y_3]) = \xi_{g},
\Psi (\overline s_\alpha) := v_\id^{\alpha}, \ \Psi (\overline v_0) := v_g^{0},
\\
p(y_1) = y_1,\quad p(y_2) = y_2, \quad p(y_3) = -y_3, \quad p(\bar{s}_\alpha) = \bar{s}_\alpha, \quad p(\bar{v}_\alpha) = - \bar{v}_\alpha.
\end{gather*}

\subsubsection*{\textbf{4.} $(D_{k}, \ZZ/2\ZZ) \sim (A_{2k-3}, \id)$ }
\begin{gather*}
f:=x_1^{2}+x_2^{k-1}+x_2x_3^2,\quad G:=\left<g\right>,\ g:=\frac{1}{2}(1,0,1), \quad \overline{f}:=y_1^{k-1}+y_1y_2+y_2y_3^2,
\\
\bar{F} (y, \bar{s}, \bar{v}) = \bar{f} + \sum_{\alpha=0}^{k-2} \bar{s}_\alpha y_1^\alpha +  \sum_{\alpha=0}^{k-3} \bar{v}_\alpha y_1^\alpha y_3, \quad F^G(s,z) = f + \sum_{\alpha=0}^{k-2} s_\alpha x_2^{\alpha},
\\
\Psi_{\bar{f}}([y_1]) = [x_2]\xi_{\id}, \ \Psi_{\bar{f}}([y_2]) = [x_3^2]\xi_{\id}, \ \Psi_{\bar{f}}([y_3]) = \xi_{g}, \
\Psi (\overline s_\alpha) := v_\id^{\alpha}, \ \Psi (\overline v_\alpha) := v_g^{\alpha},
\\
p(y_1) = y_1\quad p(y_2) = y_2 \quad p(y_3) = -y_3 \quad p(\bar{s}_\alpha) = \bar{s}_\alpha\quad p(\bar{v}_\alpha) = - \bar{v}_\alpha.
\end{gather*}

\subsubsection*{\textbf{5.} $(A_{4k+1}, \ZZ/2\ZZ) \sim (D_{2(k+1)}, \id)$ }
\begin{gather*}
f:=x_1^{2}+x_2^2+x_2x_3^{2k+1},\quad G:=\left<g\right>,\ g:=\frac{1}{2}(0,1,1), \quad \overline{f}:=y_1^2+y_2^2y_3+y_2y_3^{k+1},
\\
\bar{F} (y, \bar{s}, \bar{v}) = \bar{f} + \sum_{\alpha = 0}^{2k} \bar{s}_\alpha y_3^\alpha + \bar{v}_1(y_2 + \frac{1}{2}y_3^k) \quad F^G(s,z) = f + \sum_{\alpha = 0}^{2k} s_\alpha x_3^{2\alpha},
\\
\Psi_{\bar{f}}([y_1]) = [x_1]\xi_{\id}, \ \Psi_{\bar{f}}([y_2]) = \xi_{g}- \frac{1}{2}[x_3^{2k}]\xi_{\id}, \ \Psi_{\bar{f}}([y_3]) = [x_3^2]\xi_{\id}, 
\\
\Psi (\overline s_\alpha) := v_\id^{\alpha}, \ \Psi (\overline v_0) := v_g^{0}.
\end{gather*}

\begin{proposition}{\label{MainProp2}}
    Let $\H_{(f,G), \id}^{def}$,$\H_{(f,G), g}^{def}$ stand for the $G$--grading $\id$ and $g$ subspaces of $\H_{(f,G)}^{def}$ respectively. Denote by $\pi_\id$ and $\pi_g$ the projections on these spaces from $\H_{(f,G)}^{def}$.
    Denote by $\pi_0$ and $\pi_1$ the projections onto $\H_{\bar{F},0}$ and $\H_{\bar{F},1}$ from $\H_{\bar F}$.
    
    Then the following diagrams commute.
    \[
    \begin{tikzcd}[sep=scriptsize]
        {\H_{\bar{F}} } &&& {\H^{def}_{(f,G)}}\\
        \\
        \\
        {\H_{\bar{F},0} } &&& {\H^{def}_{(f,G), \id}}
        \arrow["{\Psi_{\bar{S}}}", from=1-1, to=1-4]
        \arrow["{\pi_0}"', from=1-1, to=4-1]
        \arrow["{\pi_\id}", from=1-4, to=4-4]
        \arrow["{\Psi_{\bar{S}}}"', from=4-1, to=4-4]
    \end{tikzcd}
    \begin{tikzcd}[sep=scriptsize]
        {\H_{\bar{F}} } &&& {\H^{def}_{(f,G)}} \\
        \\
        \\
        {\H_{\bar{F},1} } &&& {\H^{def}_{(f,G), g}}
        \arrow["{\Psi_{\bar{S}}}"', from=1-1, to=1-4]
        \arrow["{\pi_1}", from=1-1, to=4-1]
        \arrow["{\pi_g}"', from=1-4, to=4-4]
        \arrow["{\Psi_{\bar{S}}}", from=4-1, to=4-4]
    \end{tikzcd}
    \begin{tikzcd}[sep=scriptsize]
        {\H_{\bar{F}} } &&& {\H^{def}_{(f,G)}} \\
        \\
        \\
        {\H_{\bar{F},id}} &&& {\H_{\id}}
        \arrow["{\Psi_{\bar{S}}}", from=1-1, to=1-4]
        \arrow["{\pi_0}"', from=1-1, to=4-1]
        \arrow["{\pi_\id}", from=1-4, to=4-4]
        \arrow["{\Psi_{\bar{S}}}"', from=4-1, to=4-4]
    \end{tikzcd}
    \]
%
  
\end{proposition}

\begin{proof}  
This follows immediately from the construction.
\end{proof}

\begin{proposition}{\label{G-resp}}
$\nabla^{(f,G)}$ respects $G$-grading. Namely,
\[
    \nabla^{(f,G)}_{X}: \H^{def}_{(f,G),h_1} \to \H^{def}_{(f,G),h_1h_2^{-1}}, \quad \forall X \in \T \mathcal{S}_{(f,G)}^{h_2}.
\]
\end{proposition}

\begin{proof}
We make use of the map $p$ for all cases of Table~\ref{table1} except the last one. For that last case the proof is given in Proposition~\ref{EF4}.

It's enough to prove the proposition for $X$ being $ \frac{\partial}{\partial v^\alpha_{\id}}$ and $\frac{\partial}{\partial v^\alpha_{g}}$.
Following the commutative diagrams of the proposition above, it suffices to prove that $\nabla^{\bar{F}}$ respects $\ZZ/2\ZZ$-grading of $\H_{\bar{F}}$, i.e.
\[
\nabla^{\bar{F}}: \T \bar{S}^{\bar{a}} \otimes  \H_{\bar{F},\bar{b}} \to \H_{\bar{F},\overline{a+b}}.
\]

The differential operator $\frac{\partial}{\partial \bar{s}_\alpha}$, multiplication by $\bar{s}_\alpha$ and multiplication by $\frac{\partial \bar{F}}{\partial \bar{s}_\alpha} = \phi_{\alpha,id}$ 
do not change the $\ZZ/2\ZZ$--grading. 
Similarly, the differential operator $\frac{\partial}{\partial \bar{v}_\alpha}$, multiplication by $\bar{v}_\alpha$  and multiplication by $\frac{\partial \bar{F}}{\partial \bar{v}_\alpha} = \phi_{\alpha,g}$ 
alter the $\ZZ/2\ZZ$--grading. 

The map $p$ commutes with $D_{\bar F}$ and we do not need to distinguish between the classes and their representatives. This completes the proof.
\end{proof}

\subsection*{Proof of part (ii) of Theorem~\ref{theorem: main}}

Let $\H^c_{\id}$ stand for the complement to $\H_\id$ in $\H^{def}_{(f,G),\id}$. Then by Proposition~\ref{G-resp}
\[
\nabla_{\frac{\partial}{\partial v^\alpha_{\id}}}^{(f,G)}[A] 
    = \left.\nabla_{\frac{\partial}{\partial v^\alpha_{\id}}}^{(f,G)}[A]\right\vert_{\H_{\id}}  
    + \left.\nabla_{\frac{\partial}{\partial v^\alpha_{\id}}}^{(f,G)}[A]\right\vert_{\H^c_{\id}}
\]
for any $A \in \H_\id$.

\begin{lemma}
    Denote $\widetilde{F}(\bar{s},y) = \left.\bar{F}(\bar{s}, \bar{v}, y)\right\vert_{\bar{v}_\alpha =0, \Psi^{-1}_{f}[\xi_g] = 0}$. Then
    \[
    \left.\nabla_{\frac{\partial}{\partial v^\alpha_{\id}}}^{(f,G)}[A]\right\vert_{\H_{\id}} 
    =  
    \Psi_{\bar{S}} \left[\nabla^{\widetilde{F}}_{\frac{\partial}{\partial \bar{s}_\alpha}}\Psi_{\mathcal{\bar{S}}}^{-1}[A]\right].
    \]
\end{lemma}
\begin{proof}
    By linearity it's enough to consider $A$ to be given by $[\phi_\beta^{\id}]$ for some $\phi_\beta^{\id} = \phi_\beta^{\id}(x)$.
\begin{align*}
    \left.\nabla_{\frac{\partial}{\partial v^\alpha_{\id}}}^{(f,G)}[\phi_\beta^{\id} \xi_{\id}]\right\vert_{\H_{\id}} 
    &= \left.\Psi_{\bar{\mathcal{S}}} \left[ \nabla^{\bar{F}}_{\frac{\partial}{\partial s_{\alpha}}} \Psi^{-1}_{\bar{\mathcal{S}}}([\phi^{\id}_{\beta}(x) \xi_{\id}]) \right]\right\vert_{\H_{\id}} 
    = \left.\Psi_{\bar{S}} \left[\nabla^{\bar{F}}_{\frac{\partial}{\partial \bar{s}_\alpha}}[\phi^{\id}_{\beta}(y)]\right\vert_{\H_{\bar{F},id}} \right] 
    \\
    &= \Psi_{\bar{S}} \left[ \left.\nabla^{\bar{F}}_{\frac{\partial}{\partial \bar{s}_\alpha}}\Psi_{\mathcal{\bar{S}}}^{-1}[\phi^{\id}_{\beta} \xi_{\id}]\right\vert_{\bar{v} =0, \Psi^{-1}_{\mathcal{\bar{S}}}[\xi_g] = 0} \right] = \Psi_{\bar{S}} \left[\nabla^{\widetilde{F}}_{\frac{\partial}{\partial \bar{s}_\alpha}}\Psi_{\mathcal{\bar{S}}}^{-1}[\phi^{\id}_{\beta} \xi_{\id}]\right].
\end{align*}
The statement follows now from the commutativity of $p$ and $D_{\overline{F}}$.
\end{proof}

Obviously $\widetilde{F}(\bar{s},y) = \Psi^{-1}_{\bar{S}} (F^G(v_{\id}^{\alpha}, z)\xi_{\id})$, then
\begin{align*}
    \Psi_{\bar{S}} \left[\nabla^{\tilde{F}}_{\frac{\partial}{\partial \bar{s}_\alpha}}\Psi_{\mathcal{\bar{S}}}^{-1}(A)\right] 
    =& \Psi_{\bar{\mathcal{S}}} \left[z^{-1}\frac{\partial \tilde{F}}{\partial \bar{s}_\alpha}\Psi_{\mathcal{\bar{S}}}^{-1}(A)\right] 
    =  \Psi_{\bar{\mathcal{S}}} \left[z^{-1}\Psi_{\mathcal{\bar{S}}}^{-1}[\frac{\partial F^G}{\partial v_{\id}^{\alpha}}]\Psi_{\mathcal{\bar{S}}}^{-1}(A)\right] 
    \\ 
    =& \left[z^{-1}\frac{\partial F^G}{\partial v_{\id}^{\alpha}}\phi^{\id}_{\beta}\right]\xi_{\id} = \nabla^{F^G}_{\frac{\partial}{\partial v_{\id}^\alpha}}(A).
\end{align*}
This completes the proof of part (ii) of Theorem~\ref{theorem: main}.

\section{Orbifold Gauss-Manin connection}\label{section: proof-2}
In this section we make full computations of the connection $\nabla^{(f,G)}$ and complete the proof of Theorem~\ref{theorem: main}. These computations can be summarized by the following proposition.

Let $(f,G)$ be as in Table~\ref{table1} and $v_\id^0,\dots,v_\id^{k_\id-1}$, $v_g^0,\dots,v_g^{k_g-1}$ be the coordinates on $\mathcal{S}_{(f,G)}$. Denote $\phi_{\beta} := \dfrac{\p F^G}{\p v_\id^{\beta}}$.

\begin{proposition}
The connection $\nabla^{(f,G)}$ on $\H_{(f,G)}^{def}$ satisfies
\begin{enumerate}
    
\item We have for any $\alpha+\beta \le k_{\id}-1$
\[
\nabla_{\frac{\partial}{\partial v^\alpha_{\id}}}^{(f,G)}[\phi_{\beta} \xi_\id] 
=  \nabla^{F^G}_{\frac{\partial}{\partial v^\alpha_\id}}[\phi_{\beta}]\xi_{\id}.
\]
For $\alpha+\beta = k_{\id}$ we have
\[\nabla_{\frac{\partial}{\partial v^\alpha_{\id}}}^{(f,G)}[\phi_{\beta} \xi_{\id}] 
    = \nabla^{F^G}_{\frac{\partial}{\partial v^\alpha_{\id}}}[\phi_\beta ]\xi_{\id} + \left[\left.\nabla^{F^G}_{\frac{\partial}{\partial v^\alpha_{\id}}}[\phi_{\beta + k_{g} - k_{\id}}](\widetilde v) \right\vert_{z=0, \Fix(g)} + \delta_{1,k_{g}} \theta v_g^{0}
\frac{\p F^G}{\p v_\id^{\alpha+\beta-k_{\id}}}\right]\xi_{g}, \]
and for $\alpha+\beta \ge k_{\id} +1$ we have
\[\nabla_{\frac{\partial}{\partial v^\alpha_{\id}}}^{(f,G)}[\phi_{\beta} \xi_{\id}] 
=  \left[ \nabla^{F^G}_{\frac{\partial}{\partial v^\alpha_{\id}}}[\phi_{\beta} ] - \delta_{1,k_{g}} \frac{\theta}{2}(v_g^{0})^2  
\frac{\p F^G}{\p v_\id^{\alpha+\beta-k_{\id}-1}} \right]   \xi_{\id} + \left.\nabla^{F^G}_{\frac{\partial}{\partial v^\alpha_{\id}}}[\phi_{\beta}] \right\vert_{z=0, \Fix(g)} (\widetilde v) \xi_{g} \] 
where  $\theta  \in \CC^*$ is a certain non-zero constant, and $\widetilde v = \widetilde v(v)$ is a certain linear change of variables.
\item
We have for any $\alpha+\beta \le k_{g}-1$
\[
\nabla_{\frac{\partial}{\partial v^\alpha_{\id}}}^{(f,G)}[\phi_{\beta}  \xi_{g}]   = \nabla^{F^G}_{\frac{\partial}{\partial v^\alpha_{\id}}}[\phi_\beta]\xi_{g}.
\]

And for $\alpha+\beta \ge k_{g} $ we have
\[\nabla_{\frac{\partial}{\partial v^\alpha_{\id}}}^{(f,G)}[\phi_{\beta} \xi_{g}] 
=  \left. \nabla^{F^G}_{\frac{\partial}{\partial v^\alpha_{\id}}}[\phi_{\beta}]\right\vert_{\Fix(g)}\xi_{g} -\left[  \left.\nabla^{F^G}_{\frac{\partial}{\partial v^\alpha_\id}}[\phi_{\beta}] (\widetilde v) \right\vert_{z=0, \Fix(g)}+ \delta_{1,k_{g}} \frac{v_g^{0} }{2} 
\frac{\p F^G}{\p v_\id^{\alpha+\beta-k_{g}}} \right]\xi_{\id}  \] 
where  $\widetilde v = \widetilde v(v)$ is a certain linear change of variables.

     \item We have for any index $\alpha$ and any $G$--grading $g$ element $[\phi \xi_{g}]$
     \[
        \nabla_{\frac{\partial}{\partial v^\alpha_{g}}}^{(f,G)}[\phi \xi_{g}] = c_g \nabla_{\frac{\partial}{\partial v^\alpha_{\id}}}^{(f,G)}[\phi H_g^F] (\widetilde v) \xi_{\id}
    \]         
    where $\widetilde v = \widetilde v(v)$ is the certain linear change of the variables, $H_g^F$ and $c_g$ as in Section~\ref{section: H_g}.
    \end{enumerate}
\end{proposition}
\begin{proof}
 Is summarized in Propositions~\ref{EF1}, ~\ref{EF2}, ~\ref{EF3}, ~\ref{EF4} and \ref{EF5}.
\end{proof}

To simplify the formulae below, we adopt the following notation for the $\H_{\bar F}$--elements
\[
    [\phi]_{\bar{F}} := [\phi dy_1\wedge dy_2 \wedge dy_3].
\]

We follow the indexing of Table~\ref{table1} and make us of the data given Section~\ref{section: proof-1}.

\subsection*{1} $f:=x_1^{k+1}+x_2^2+x_3^2$, $G:=\left<g\right>$, $g:=\frac{1}{2}(0,1,1)$, $\overline{f}:=y_1^{k+1}+y_2+y_2y_3^2$.

\begin{proposition}\label{EF1}
    For the pair $(f,G)$ as above, the action of orbifold Gauss-Manin connection reads
\begin{enumerate}
\item 
If $\alpha+\beta \le k-1$, then
\[
\nabla_{\frac{\partial}{\partial v^\alpha_{\id}}}^{(f,G)}[x_1^{2\beta}\xi_{\id}] = \nabla^{F^G}_{\frac{\partial}{\partial v^\alpha_{\id}}}[x_1^{2\beta}]\xi_{\id}.
\]
If $\alpha+\beta \ge k$, then
\[
\nabla_{\frac{\partial}{\partial v^\alpha_{\id}}}^{(f,G)}[x_1^{\beta}\xi_{\id}] = \nabla^{F^G}_{\frac{\partial}{\partial v^\alpha_{\id}}}[x_1^{\beta}]\xi_{\id} + \left.\nabla^{F^G}_{\frac{\partial}{\partial v^\alpha_{\id}}}[x_1^{\beta}]\xi_{g}(\widetilde v)\right\vert_{z=0},
\]
where $\widetilde v^{\alpha}_\id = v_g^{\alpha}$.
\item 
If $\alpha+\beta \le k-1$, then
\[
\nabla_{\frac{\partial}{\partial v^\alpha_{g}}}^{(f,G)}[x_1^{2\beta}\xi_{\id}] = \nabla^{F^G}_{\frac{\partial}{\partial v^\alpha_{\id}}}[x_1^{2\beta}]\xi_{g}.
\]
If $\alpha+\beta \ge k$, then
\[
\nabla_{\frac{\partial}{\partial v^\alpha_{g}}}^{(f,G)}[x_1^{\beta}\xi_{\id}] = \nabla_{\frac{\partial}{\partial v^\beta_{\id}}}^{(f,G)}[x_1^{\alpha}\xi_{g}]  = \nabla^{F^G}_{\frac{\partial}{\partial v^\alpha_{\id}}}[x_1^{\beta}]\xi_{g} - \left.\nabla^{F^G}_{\frac{\partial}{\partial v^\alpha_{\id}}}[x_1^{\beta}]\xi_{\id}\right\vert_{ z=0},
\]
where $\widetilde v^{\alpha}_\id = v_g^{\alpha}$.

\item \[
\nabla_{\frac{\partial}{\partial v^\alpha_{g}}}^{(f,G)}[x_1^\beta\xi_{g}] =  c_g\nabla_{\frac{\partial}{\partial v^\alpha_{\id}}}^{(f,G)}[x_1^\beta  H_g^F]\xi_{\id}.
\]
\end{enumerate}
\end{proposition}

\begin{proof}

 (i) By the results of the previous section we have
\[
\nabla_{\frac{\partial}{\partial v^\alpha_{\id}}}^{(f,G)}[x_1^{\beta}\xi_{\id}] = \nabla^{F^G}_{\frac{\partial}{\partial v^\alpha_{\id}}}[x_1^{\beta}]\xi_{\id} + \left.\nabla_{\frac{\partial}{\partial v^\alpha_{\id}}}^{(f,G)}[x_1^{\beta}\xi_{\id}]\right\vert_{\H^c_{\id}}.
\]
Now we calculate the second term of $\nabla_{\frac{\partial}{\partial v^\alpha_{\id}}}^{(f,G)}[x_1^{\beta}\xi_{\id}]$ via the map $\Psi_{\bar{S}}$:
\[
\nabla_{\frac{\partial}{\partial v^\alpha_{\id}}}^{(f,G)}[x_1^{\beta}\xi_{\id}] = \Psi_{\bar{S}}[y_1^{\alpha+\beta}].
\]
If $ 0\le \alpha+\beta \le k-1$, $\Psi_{\bar{S}} [y_1^{\alpha+\beta}]_{\bar{F}} = [x_1^{\alpha+\beta}]\xi_{\id} = \nabla^{F^G}_{\frac{\partial}{\partial v^\alpha_\id}}[x_1^{\beta}]\xi_{\id}$. For $ k\le \alpha+\beta \le 2(k-1)$ we need to take into account the relations in the Brieskorn lattice $\H_{\bar F}$. Denote $p := \alpha+\beta - k$. Then $0 \le p \le k-2$. Applying $D_{\bar{F}}$ to $[y_1^{p}dy_2\wedge dy_3]$  we obtain
\[
[y_1^{\alpha+\beta}]_{\bar{F}} = -\frac{1}{k+1}[z p y_1^{p-1} + \sum^{k-1}_{\gamma = 0} \gamma  y_1^{\gamma + p -1}(\bar{s}_\gamma + \bar{v}_\gamma y_3) ]_{\bar{F}}.
\]
The restriction to $\H^c_{\id}$ gives us
\[
\left.[y_1^{\alpha+\beta}]_{\bar{F}}\right\vert_{\H_{\id}^c} = -\frac{1}{k+1}[\sum_{\gamma = 0}^{k-1} \gamma  y_1^{\gamma + p -1}\bar{v}_\gamma y_3]_{\bar{F}}
\]
from where we obtain the claim.

(ii) By the definition $\nabla_{\frac{\partial}{\partial v^\beta_\id}}^{(f,G)}[x_1^{\alpha}\xi_g] = \nabla_{\frac{\partial}{\partial v^\alpha_{g}}}^{(f,G)}[x_1^{\beta}\xi_\id]  = \Psi_{\bar{S}}[y_1^{\alpha+\beta}y_3]_{\bar{F}}$.
If $\alpha+\beta \le k-1$, $\Psi_{\bar{S}}[y_1^{\alpha+\beta}y_3]_{\bar{F}} = [x_2^{\alpha+\beta}] \xi_g = \nabla^{F^G}_{\frac{\partial}{\partial v^\alpha_\id}}[x_2^{\beta}]\xi_{g}$. For $\alpha+\beta \ge k$ we need to take into account the relations in the Brieskorn lattice $\H_{\bar F}$. Applying $D_{\bar{F}}$ to $[y_1^{p}y_3dy_2\wedge dy_3]$ and $[-y_1^{p-1+\gamma}dy_1\wedge dy_3]$ we obtain
\[
[y_1^{\alpha+\beta}]_{\bar{F}} = -\frac{1}{k+1}[z p y_1^{p-1}y_3 + \sum_{\gamma = 0}^{k-1} \gamma\bar{s}_\gamma   y_1^{p-1+\gamma}y_3 + \sum^{k-1}_{\gamma =0} \gamma  \bar{v}_\gamma y_1^{p-1+\gamma}y_3^2 ] _{\bar{F}},
\]
\[
[y_1^{p-1+\gamma} + y_1^{p-1+\gamma}y_3^2]_{\bar{F}} = 0.
\]
These relations give us
\[
[y_1^{\alpha+\beta}]_{\bar{F}} = -\frac{1}{k+1}[(z p y_1^{p-1} + \sum_{\gamma = 0}^{k-1} \gamma\bar{s}_\gamma   y_1^{p-1+\gamma})y_3 - \sum^{k-1}_{\gamma =0} \gamma  \bar{v}_\gamma y_1^{p-1+\gamma} ] _{\bar{F}}.
\]
Moreover,
\[ 
\nabla_{\frac{\partial}{\partial v^\alpha_{\id}}}^{(f,G)}[x_1^{\beta}\xi_g] = \Psi_{\bar{S}}[y_1^{\alpha + \beta}y_3]_{\bar{F}} = 
\nabla_{\frac{\partial}{\partial v^\beta_{g}}}^{(f,G)}[x_1^{\alpha}\xi_{\id}]
\]
what proves (ii).

(iii) $\nabla_{\frac{\partial}{\partial v^\alpha_{g}}}^{(f,G)}[x_1^\beta\xi_{g}] = \Psi_{\bar{S}}[y_1^{\alpha+\beta}y_3^2]_{\bar{F}}$.
Applying $D_{\bar{F}}$ to $[-y_1^{\alpha+\beta}dy_1\wedge dy_3]$, we obtain:
\[
[y_1^{\alpha+\beta} + y_1^{\alpha+\beta}y_3^2]_{\bar{F}} = 0
\]
from where we conclude
\[
\nabla_{\frac{\partial}{\partial v^\alpha_{g}}}^{(f,G)}[x_1^\beta\xi_{g}] = - \Psi_{\bar{S}}[y_1^{\alpha+\beta}]_{\bar{F}} = -\nabla_{\frac{\partial}{\partial v^\alpha_{\id}}}^{(f,G)}[x_1^{\beta}\xi_{\id}].
\]
Note that $c_gH_{g}^F = -\frac{1}{|G|^2}H_{g}^F = -1$ what completes the proof.
\end{proof}

\subsection*{2}
$f:=x_1^{2k}+x_2^2+x_3^2$, $G:=\left<g\right>$, $g:=\frac{1}{2}(1,0,1)$, $\overline{f}:=y_1^2+y_2^k+y_2y_3^2$.
%

\begin{proposition}{\label{EF2}}
    For the pair $(f,G)$ as above, the action of orbifold Gauss-Manin connection reads
\begin{enumerate}
\item 
If $\alpha+\beta \le k-1$, then
\[
\nabla_{\frac{\partial}{\partial v^\alpha_{\id}}}^{(f,G)}[x_1^{2\beta}\xi_{\id}] = \nabla^{F^G}_{\frac{\partial}{\partial v^\alpha_{\id}}}[x_1^{2\beta}]\xi_{\id},
\]
If $\alpha+\beta = k$, then
\[
\nabla_{\frac{\partial}{\partial v^\alpha_{\id}}}^{(f,G)}[x_1^{2\beta}\xi_{\id}] = \nabla^{F^G}_{\frac{\partial}{\partial v^\alpha_{\id}}}[x_1^{2\beta}]\xi_{\id} + \frac{v_g^0}{2k} \left[\frac{\p F^G}{\p v_\id^{\alpha+\beta-k}}\right]
\xi_{g},
\]
If $\alpha+\beta > k$, then
\[
\nabla_{\frac{\partial}{\partial v^\alpha_{\id}}}^{(f,G)}[x_1^{2\beta}\xi_{\id}] = \nabla^{F^G}_{\frac{\partial}{\partial v^\alpha_{\id}}}[x_1^{2\beta}]\xi_{\id} - \frac{(v_g^0)^2}{4k} 
\left[\frac{\p F^G}{\p v_\id^{\alpha+\beta-k-1}}\right]\xi_{\id},
\]

\item If $1 \le \beta \le k-1$, then we have
\[
\nabla_{\frac{\partial}{\partial v^0_{g}}}^{(f,G)}[x_1^{2\beta}\xi_{\id}] 
= \nabla_{\frac{\partial}{\partial v^\beta_{\id}}}^{(f,G)}[\xi_{g}] 
= - \frac{v_g^0}{2}\left[\frac{\partial F^G}{\partial v^{\beta-1}_{\id}}\right]\xi_{\id},
\]
and
\[
\nabla_{\frac{\partial}{\partial v^0_{g}}}^{(f,G)}[1 \cdot \xi_{\id}] 
= \nabla_{\frac{\partial}{\partial v^0_{\id}}}^{(f,G)}[\xi_{g}] 
= \xi_g,
\]

\item 
    \[
        \nabla_{\frac{\partial}{\partial v^0_{g}}}^{(f,G)}[\xi_{g}] = 
        c_g \nabla_{\frac{\partial}{\partial v^0_{\id}}}^{(f,G)}[H^F_g \xi_{\id}](\widetilde v_{\id}) = c_g[H^F_g (\widetilde v_{id})]\xi_{\id}
    \]
where $\widetilde v^{\alpha}_{id} := v^{\alpha}_{id} \frac{2k - 1}{2\alpha -1}$.  
\end{enumerate}
\end{proposition}

\begin{proof}

 (i) By the results of the previous section we have
\[
 \nabla_{\frac{\partial}{\partial v^\alpha_{\id}}}^{(f,G)}[x_1^{2\beta}\xi_{\id}] = \nabla^{F^G}_{\frac{\partial}{\partial v^\alpha_{\id}}}[x_1^{2\beta}]\xi_{\id} + \left.\nabla_{\frac{\partial}{\partial v^\alpha_{\id}}}^{(f,G)}[x_1^{2\beta}\xi_{\id}]\right\vert_{\H_{\id}^c}.
\]
Now we calculate the second term of $\nabla_{\frac{\partial}{\partial v^\alpha_{\id}}}^{(f,G)}[x_1^{2\beta}]\xi_{\id}$ via the map $\Psi_{\bar{S}}$:
\[
\nabla_{\frac{\partial}{\partial v^\alpha_{\id}}}^{(f,G)}[x_1^{2\beta}\xi_{\id}] = \Psi_{\bar{S}}[y_2^{\alpha+\beta}]_{\bar{F}}.
\]
If $ 0 \le \alpha+\beta \le k-1$, $\Psi_{\bar{S}} [y_2^{\alpha+\beta}]_{\bar{F}} = [x_1^{2(\alpha+\beta)}]\xi_{\id} = \nabla^{F^G}_{\frac{\partial}{\partial v^\alpha_\id}}[x_1^{2\beta}]\xi_{\id}$. For $k \le \alpha+\beta \le 2k-2$ we need to take into account the relations in the Brieskorn lattice $\H_{\bar F}$. Denote $p := \alpha+\beta - k$. Then $0 \le p \le k-2$. Applying $D_{\bar{F}}$ to $[y_2^{p+1}dy_1\wedge dy_3]$ and $[y_2^{p}y_3dy_1\wedge dy_2]$ we have
 \[
 [z (p+1)y_2^{p} + ky_2^{\alpha +\beta} + y_2^{p+1}y_3^2 +\sum_{\gamma = 0}^{k-1}\gamma\bar{s}_\gamma y_2^{\gamma + p }]_{\bar{F}} = [0]_{\bar{F}},
 \]
 \[
 [z y_2^{p} + 2y_2^{p+1}y_3^2 + \bar{v}_{0}y_2^{p}y_3]_{\bar{F}} = [0] _{\bar{F}}.
 \]
These relations give us
 \[
 [y_2^{\alpha+\beta}]_{\bar{F}} = -\frac{1}{k}[z(p + \frac{1}{2}) y_2^{p-1}+\sum_{\gamma = 0}^{k-1}\gamma\bar{s}_\gamma y_2^{\gamma +p}]_{\bar{F}} + \frac{\bar{v}_{0}}{2k}[y_2^{p}y_3]_{\bar{F}}
 \]
 If $p > 0$, we need to express the last summand as well. Applying $D_{\bar F}$ to $[y_2^{p-1} dy_1 \wedge dy_2]$, we have
 \[
    [2 y_2^py_3 + \bar v_0 y_2^{p-1}]_{\bar F} = 0
 \]
 what completes the proof.

(ii) $\nabla_{\frac{\partial}{\partial v^0_{g}}}^{(f,G)}[x_1^{2\beta}\xi_{\id}] = \Psi_{\bar{S}}[y_2^{\beta}y_3]_{\bar{F}}$. If $\beta = 0$, the statement is clear. If $1\le \beta \le k-1$, applying $D_{\bar{F}}$ to $[y_2^{\beta-1}dy_1\wedge dy_2]$ we get
\[
[2y_2^\beta y_3 + \bar{v}_0y_2^{\beta-1}]_{\bar{F}} = 0.
\]
We calculate
\[
\nabla_{\frac{\partial}{\partial v^0_{g}}}^{(f,G)}[x_1^{2\beta}\xi_{\id}] = \Psi_{\bar{S}}[y_2^{\beta}y_3]_{\bar{F}} 
= \Psi_{\bar{S}}[-\frac{\bar{v}_0}{2}y_1^{\beta-1}]_{\bar{F}} 
= - \frac{v_g^0}{2}[x_1^{2\beta-2}]\xi_{\id} 
=  - \frac{v_g^0}{2}\left[\frac{\partial F^G}{\partial v^{\beta-1}_{\id}}\right]\xi_{\id}.
\]
Moreover
\[ 
\nabla_{\frac{\partial}{\partial v^\alpha_{\id}}}^{(f,G)}[\xi_g] = \Psi_{\bar{S}}[y_2^{\alpha}y_3]_{\bar{F}} = 
\nabla_{\frac{\partial}{\partial v^0_g}}^{(f,G)}[x_1^{2\alpha}\xi_{\id}]
\]
what proves (ii).

(iii) $\nabla_{\frac{\partial}{\partial v^0_{g}}}^{(f,G)}[\xi_{g}] = \Psi_{\bar{S}}[y_3^2]_{\bar{F}}$. Applying $D_{\bar{F}}$ to $[-dy_1\wedge dy_3]$ we obtain:
\[
[ky_2^{k-1} + y_3^2 + \sum \alpha \bar{s}_{\alpha} y_2^{\alpha-1}]_{\bar{F}} = 0
\]
from where we calculate
\[
\nabla_{\frac{\partial}{\partial v^0_{g}}}^{(f,G)}[\xi_{g}] = \Psi_{\bar{S}}[y_3^2]_{\bar{F}} 
= \Psi_{\bar{S}}[ -(ky_2^{k-1} + \sum \alpha \bar{s}_{\alpha} y_2^{\alpha-1})]_{\bar{F}} 
= -[kx_1^{2k-2} + \sum \alpha v^{\alpha}_{\id} x_1^{2\alpha-2}]\xi_{\id}.
\]

Note that
\[
    c_gH^F_g = -\frac{H^F_g}{|G|^2} = -\frac{1}{4}\frac{1}{2k-1}(4k(2k-1)x_1^{2k-2} +  \sum^{k-1}_{\alpha = 0} 4\alpha(2\alpha-1)v_{\id}^{\alpha} x_1^{2\alpha-2}) 
\]
what completes the proof.
\end{proof}

\subsection*{3} 
$f:=x_1^{2}+x_2^2+x_2x_3^{2k}$, $G:=\left<g\right>$, $g:=\frac{1}{2}(1,0,1)$,  $\overline{f}:=y_1^2+y_1y_2^k+y_2y_3^2$.


\begin{proposition}{\label{EF3}}
    For the pair $(f,G)$ as above, the action of orbifold Gauss-Manin connection reads
\begin{enumerate}
\item 
If $\alpha+\beta \le 2k-1$, then
\[
\nabla_{\frac{\partial}{\partial v^\alpha_{\id}}}^{(f,G)}[x_3^{2\beta}\xi_{\id} ] = \nabla^{F^G}_{\frac{\partial}{\partial v^\alpha_{\id}}}[x_3^{2\beta}]\xi_{\id}
\]
If $\alpha+\beta = 2k$, then
\begin{gather*}
\nabla_{\frac{\partial}{\partial v^\alpha_{\id}}}^{(f,G)}[x_3^{2\beta}\xi_{\id} ]  =  \nabla^{F^G}_{\frac{\partial}{\partial v^\alpha_{\id}}}[x_3^{2\beta}]\xi_{\id} + \frac{v_g^0}{k}\left[\frac{\p F^G}{\p v_\id^{\alpha+\beta-2k}}\right] \xi_{g}.
\end{gather*}
If $\alpha+\beta > 2k$, then
\begin{gather*}
\nabla_{\frac{\partial}{\partial v^\alpha_{\id}}}^{(f,G)}[x_3^{2\beta}\xi_{\id} ]  =  \nabla^{F^G}_{\frac{\partial}{\partial v^\alpha_{\id}}}[x_3^{2\beta}]\xi_{\id} - \frac{(v_g^0)^2}{2k} 
\left[\frac{\p F^G}{\p v_\id^{\alpha+\beta-2k-1}}\right]\xi_{\id}.
\end{gather*}
\item If $1 \le \beta \le 2k-1$, then we have
\[
\nabla_{\frac{\partial}{\partial v^0_{g}}}^{(f,G)}[x_3^{2\beta}\xi_{\id}] 
= \nabla_{\frac{\partial}{\partial v^\beta_{\id}}}^{(f,G)}[\xi_{g}] 
= - \frac{v_g^0}{2}\left[\frac{\partial F^G}{\partial v^{\beta-1}_{\id}}\right]\xi_{\id},
\]
and
\[
\nabla_{\frac{\partial}{\partial v^0_{g}}}^{(f,G)}[1 \cdot \xi_{\id}] 
= \nabla_{\frac{\partial}{\partial v^0_{\id}}}^{(f,G)}[\xi_{g}] 
= \xi_g,
\]

\item \[
        \nabla_{\frac{\partial}{\partial v^0_{g}}}^{(f,G)}[\xi_{g}] = 
        c_g \nabla_{\frac{\partial}{\partial v^0_{\id}}}^{(f,G)}[H^F_g \xi_{\id}](\widetilde v_{id}) = c_g[H^F_g (\widetilde v_{id})]\xi_{\id}
    \]
where $\widetilde v^{\alpha}_{id} := v^{\alpha}_{id} \frac{2k - 1}{2\alpha -1}$.  
\end{enumerate}
\end{proposition}

\begin{proof}

 (i) By the results of the previous section we have
\[
 \nabla_{\frac{\partial}{\partial v^\alpha_{\id}}}^{(f,G)}[x_3^{2\beta}\xi_{\id}] = \nabla^{F^G}_{\frac{\partial}{\partial v^\alpha_{\id}}}[x_3^{2\beta}]\xi_{\id} + \left.\nabla_{\frac{\partial}{\partial v^\alpha_{\id}}}^{(f,G)}[x_3^{2\beta}\xi_{\id}]\right\vert_{\H_{\id}^c}.
\]
Now we calculate the second term of $\nabla_{\frac{\partial}{\partial v^\alpha_{\id}}}^{(f,G)}[x_3^{2\beta}]\xi_{\id}$ via the map $\Psi_{\bar{S}}$:
\[
\nabla_{\frac{\partial}{\partial v^\alpha_{\id}}}^{(f,G)}[x_3^{2\beta}\xi_{\id}] = \Psi_{\bar{S}}[y_2^{\alpha+\beta}]_{\bar{F}}.
\]
If $0 \le \alpha+\beta \le 2k-1$, $\Psi_{\bar{S}} [y_2^{\alpha+\beta}]_{\bar{F}} = [x_3^{2(\alpha+\beta)}]\xi_{\id} = \nabla^{F^G}_{\frac{\partial}{\partial v^\alpha_\id}}[x_3^{2\beta}]\xi_{\id}$. For $2k \le \alpha+\beta \le 4k-2$ we need to take into account the relations in the Brieskorn lattice $\H_{\bar F}$. Denote $p := \alpha+\beta - 2k$. Then $0 \le p \le 2k-2$. Applying $D_{\bar{F}}$ to $[y_2^{p+k}dy_2\wedge dy_3]$, $[-y_2^{p+1}y_3dy_1\wedge dy_3]$, and $[y_2^{p}y_3 dy_1 \wedge dy_2]$, we have
\[
 [2y_1y_2^{p+k} + y_2^{\alpha+\beta}]_{\bar{F}} = 0,
 \]
 \[
 [z(p+1)y_2^{p} + ky_1y_2^{p+k} + y_3^{2}y_2^{p+1} + \sum_{\gamma = 0}^{2k-1}\gamma \bar{s}_\gamma y_2^{p+\gamma }] = 0,
 \]
 \[
 [z y_2^{p} +2y_2^{p+1}y_3^2 + \bar{v}_0y_2^{p}y_3] = 0.
 \]
 These relations give us
 \[
 [y_2^{\alpha+\beta}]_{\bar{F}} = \frac{2}{k}[z(p + \frac{1}{2}) y_2^{p}+\sum_{\gamma = 0}^{2k-1}\gamma\bar{s}_\gamma y_2^{\gamma +p}]_{\bar{F}} + \frac{\bar{v}_{0}}{k}[y_2^{p}y_3]_{\bar{F}}
 \]
If $p > 0$, we need to express the last summand as well. Applying $D_{\bar F}$ to $[y_2^{p-1}dy_1\wedge dy_2]$, we have
  \[
 [2y_2^{p}y_3 + \bar{v}_0y_2^{p-1}] = 0.
 \]
 what completes the proof.

(ii) If $0\leq \beta \leq 2k-1$, then $\nabla_{\frac{\partial}{\partial v^0_{g}}}^{(f,G)}[x_3^{2\beta}\xi_{\id}] = \Psi_{\bar{S}}[y_2^{\beta}y_3]_{\bar{F}}$. Applying $D_{\bar{F}}$ to $[y_2^{\beta-1}dy_1\wedge dy_2]$ we obtain:
\[
[2y_2^\beta y_3 + \bar{v}_0 y_2^{\beta-1}]_{\bar{F}} = 0
\]
from where we calculate
\[
\nabla_{\frac{\partial}{\partial v^1_{g}}}^{(f,G)}[x_3^{2\beta}\xi_{\id}] = \Psi_{\bar{S}}[y_2^{\beta}y_3]_{\bar{F}} = \Psi_{\bar{S}}[-\frac{\bar{v}_0}{2}y_2^{\beta-1}]_{\bar{F}} = - \frac{v_g^0}{2}[x_3^{2\beta-2}]\xi_{\id} =  - \frac{v_g^0}{2}\left[\frac{\partial F^G}{\partial v^{\beta-1}_{\id}}\right]\xi_{\id}.
\]
Moreover,
\[ 
\nabla_{\frac{\partial}{\partial v^\alpha_{\id}}}^{(f,G)}[\xi_g] = \Psi_{\bar{S}}[y_2^{\alpha}y_3]_{\bar{F}} = 
\nabla_{\frac{\partial}{\partial v^0_{g}}}^{(f,G)}[x_3^{2\alpha}\xi_{\id}]  
\]
what proves (ii).

(iii) $\nabla_{\frac{\partial}{\partial v^0_{g}}}^{(f,G)}[\xi_{g}] = \Psi_{\bar{S}}[y_3^2]$. Applying $D_{\bar{F}}$ to $[-dy_1\wedge dy_3]$ we obtain:
\[
[ky_1y_2^{k-1} + y_3^2 + \sum_{\alpha = 0}^{2k-1} \alpha \bar{s}_{\alpha} y_2^{\alpha-1}]_{\bar{F}} = 0
\]
from where we calculate
\begin{align*}
    \nabla_{\frac{\partial}{\partial v^0_{g}}}^{(f,G)}[\xi_{g}] &= \Psi_{\bar{S}}[y_3^2] = - \Psi_{\bar{S}}[ky_1y_2^{k-1} + \sum_{\alpha = 0}^{2k-1} \alpha \bar{s}_{\alpha} y_2^{\alpha-1}] _{\bar{F}} = -[kx_3^{2k-2}x_2 + \sum_{\alpha=0}^{2k-1} \alpha v_{\alpha}^{\id} x_3^{2\alpha-2}]\xi_{\id} .
\end{align*}

Note that
\begin{align*}
    c_gH^F_g = -\frac{H^F_g}{|G|^2} &= \frac{1}{4} \frac{1}{2k-1} \left(4k(2k-1)x_2x_3^{2k-2} 
    +  \sum^{k-1}_{\alpha = 0} 4\alpha(2\alpha-1)v_{\alpha}^{\id} x_3^{2\alpha-2} \right)
\end{align*}
what completes the proof.
\end{proof}

\subsection*{4}
$f := x_1^{2} + x_2^{k-1} + x_2x_3^{2}$, $G:=\left<g\right>$, $g:=\frac{1}{2}(1,0,1)$, $\overline{f} := y_1^{k-1} + y_1y_2 + y_2y_3^2$.
%

\begin{proposition}{\label{EF5}}
    For the pair $(f,G)$ as above, the action of orbifold Gauss-Manin connection reads
\begin{enumerate}
\item
If $\alpha+\beta \le k-2$, then
\[
\nabla_{\frac{\partial}{\partial v^\alpha_{\id}}}^{(f,G)}[x_2^{\beta}\xi_{\id}] = \nabla^{F^G}_{\frac{\partial}{\partial v^\alpha_{\id}}}[x_2^{\beta}]\xi_{\id},
\]
If $\alpha+\beta \ge k-1$, then
\[\nabla_{\frac{\partial}{\partial v^\alpha_{\id}}}^{(f,G)}[x_2^{\beta}\xi_{\id}] = \nabla^{F^G}_{\frac{\partial}{\partial v^\alpha_{\id}}}[x_2^{\beta}]\xi_{\id}  +\left.\nabla^{F^G}_{\frac{\partial}{\partial v^\alpha_\id}}[x_2^{\beta-1}] (\widetilde v) \xi_{g}\right\vert_{z=0}
\]
where $\widetilde v^{\alpha+1}_\id = v_g^{\alpha} (\alpha + \frac{1}{2})$.

\item If $\alpha+\beta \le k-3$, then
\[
\nabla_{\frac{\partial}{\partial v^\beta_\id}}^{(f,G)}[x_2^{\alpha}\xi_g] = \nabla_{\frac{\partial}{\partial v^\alpha_{g}}}^{(f,G)}[x_2^{\beta}\xi_\id] =  \nabla^{F^G}_{\frac{\partial}{\partial v^\alpha_\id}}[x_2^{\beta}]\xi_{g}.
\]
If $\alpha+\beta \ge k-2$, then
\[
\nabla_{\frac{\partial}{\partial v^\beta_\id}}^{(f,G)}[x_2^{\alpha}\xi_g] = \nabla_{\frac{\partial}{\partial v^\alpha_{g}}}^{(f,G)}[x_2^{\beta}\xi_\id] = \nabla^{F^G}_{\frac{\partial}{\partial v^\alpha_\id}}[x_2^{\beta}]\xi_{g} 
- \left.\nabla^{F^G}_{\frac{\partial}{\partial v^\alpha_\id}}[x_2^{\beta}] (\widetilde v) \xi_\id\right\vert_{z=0}
\]
where $\widetilde v^{\alpha+1}_\id = v_g^{\alpha} (\alpha + \frac{1}{2})$.
\item 
\[
    \nabla_{\frac{\partial}{\partial v^\alpha_{g}}}^{(f,G)}[x_2^\beta\xi_{g}] = c_g\nabla_{\frac{\partial}{\partial v^\alpha_{\id}}}^{F^G}[x_2^\beta  H_g^F]\xi_{\id}..
\]
\end{enumerate}
\end{proposition}

\begin{proof}

(i) By the results of the previous section we have
\[
\nabla_{\frac{\partial}{\partial v^\alpha_{\id}}}^{(f,G)}[x_2^{\beta}\xi_{\id}] = \nabla^{F^G}_{\frac{\partial}{\partial v^\alpha_{\id}}}[x_2^{\beta}]\xi_{\id} + \left.\nabla_{\frac{\partial}{\partial v^\alpha_{\id}}}^{(f,G)}[x_2^{\beta}\xi_{\id}]\right\vert_{\H^c_{\id}}.
\]
Now we calculate the second term of $\nabla_{\frac{\partial}{\partial v^\alpha_{\id}}}^{(f,G)}[x_1^{\beta}]\xi_{\id}$ via the map $\Psi_{\bar{S}}$:
\[
\nabla_{\frac{\partial}{\partial v^\alpha_{\id}}}^{(f,G)}[x_2^{\beta}\xi_{\id}] = \Psi_{\bar{S}}[y_1^{\alpha+\beta}]_{\bar{F}}.
\]
If $0 \le \alpha+\beta \le k-2$, $\Psi_{\bar{S}} [y_1^{\alpha+\beta}]_{\bar{F}} = [x_2^{\alpha+\beta}]\xi_{\id} = \nabla^{F^G}_{\frac{\partial}{\partial v^\alpha_\id}}[x_2^{\beta}]\xi_{\id}$. For $k-1 \le\alpha+\beta \le 2k-4$ we need to take into account the relations in the Brieskorn lattice $\H_{\bar F}$. Denote $p := \alpha+\beta - (k-1)$. Then $0 \le p \le k-3$. Applying $D_{\bar{F}}$ to $[y_1^{p+1}dy_2\wedge dy_3]$, $-[y_1^{p}y_2dy_1 \wedge dy_3]$, and $[y_1^{p}y_3dy_1\wedge dy_2]$, we obtain
\[
[z (p+1) y_1^{p} + (k-1)y_1^{\alpha+\beta} + y_1^{p+1}y_2 + \sum_{\gamma=0}^{k-2} \gamma  y_1^{\gamma +p}(\bar{s}_\gamma + \bar{v}_\gamma y_3)]_{\bar{F}} = 0,
\]
\[
[z y_1^{p} + y_1^{p+1}y_2 + y_1^{p}y_2y_3^2] _{\bar{F}}= 0,
\]
\[
[z y_1^{p} + 2y_1^{p}y_2y_3^2 + \sum_{\gamma =0}^{k-3}\bar{v}_{\gamma}y_1^{p+\gamma}y_3] _{\bar{F}} = 0
\]
where we assume $\bar{v}_{k-2} \equiv 0$. These relations give us
\[
[y_1^{\alpha+\beta}]_{\bar{F}} = -\frac{1}{k-1}[z (p +\frac{1}{2}) y_1^{p} + \sum_{\gamma=0}^{k-2} \gamma \bar{s}_\gamma y_1^{\gamma +p}]_{\bar{F}}  -\frac{1}{k-1} [\sum_{\gamma=0}^{k-3} (\gamma + \frac{1}{2})\bar{v}_\gamma y_1^{\gamma+p} y_3)]_{\bar{F}}.
\]
This completes the proof.

(ii) By the definition $\nabla_{\frac{\partial}{\partial v^\beta_\id}}^{(f,G)}[x_2^{\alpha}\xi_g] = \nabla_{\frac{\partial}{\partial v^\alpha_{g}}}^{(f,G)}[x_2^{\beta}\xi_\id]  = \Psi_{\bar{S}}[y_1^{\alpha+\beta}y_3]_{\bar{F}}$.
If $\alpha+\beta \le k-3$, $\Psi_{\bar{S}}[y_1^{\alpha+\beta}y_3]_{\bar{F}} = [x_2^{\alpha+\beta}] \xi_g = \nabla^{F^G}_{\frac{\partial}{\partial v^\alpha_\id}}[x_2^{\beta}]\xi_{g}$. For $\alpha+\beta \ge k-2$ we need to take into account the relations in the Brieskorn lattice $\H_{\bar F}$. Denote $r := \alpha+\beta - (k-2)$. Then $0 \le r \le k-3$.

Applying $D_{\bar{F}}$ to $[y_1^{r}y_3 dy_2\wedge dy_3]$, $[y_1^{r} dy_1 \wedge dy_2]$ and $[-y_1^{r+k-2} dy_1 \wedge dy_3]$ we have
\begin{align*}
    & [z r y_1^{r-1}y_3 + (k-1)y_1^{k-2+r}y_3 + y_1^{r}y_2y_3 + \sum_{\gamma=0}^{k-2} \gamma(\bar s_\gamma + \bar v_\gamma y_3) y_1^{r-1+\gamma}y_3]_{\bar{F}} = 0,
    \\
    & [2y_1^{r}y_2y_3 + \sum_{\gamma=0}^{k-3} \bar{v}_\gamma y_1^{\gamma +r}]_{\bar{F}} = 0, \quad [y_1^{r+k-1} + y_1^{r+k-2} y_3^2]_{\bar{F}}  = 0
\end{align*}
where we assume $\bar{v}_{k-2} \equiv 0$.
These relations give us
\begin{align*}
    &[y_1^{\alpha + \beta}y_3]_{\bar{F}} = -\frac{1}{k-1} [(z r y_1^{r-1} + \sum_{\gamma=0}^{k-2} \gamma \bar s_\gamma y_1^{r-1+\gamma} )y_3
    - \sum_{\gamma=0}^{k-3} \frac{\bar v_{\gamma}}{2} y_1^{r+\gamma} 
     - \sum_{\gamma=0}^{k-3} \gamma \bar v_\gamma y_1^{r+\gamma}]_{\bar{F}}.
\end{align*}
This completes the proof.

%

(iii) $\nabla_{\frac{\partial}{\partial v^\alpha_{g}}}^{(f,G)}[x_2^\beta\xi_{g}] = \Psi_{\bar{S}}[y_1^{\alpha+\beta}y_3^2]$.
Applying $D_{\bar{F}}$ to $[y_1^{\alpha+\beta}dy_1\wedge dy_3]$, we get:
\[
[y_1^{\alpha+\beta+1} + y_1^{\alpha+\beta}y_3^2]_{\bar{F}} = 0
\]
from what we conclude
\[
\nabla_{\frac{\partial}{\partial v^\alpha_{g}}}^{(f,G)}[x_2^\beta\xi_{g}] = - [x_2^{\alpha+\beta+1}\xi_{\id}].
\]
Note that $c_gH_{g}^F = -\frac{1}{|G|^2}H_{g}^F = -x_2$ what completes the proof.
\end{proof}

\subsection*{5} 
$f:=x_1^{2}+x_2^2+x_2x_3^{2k+1}$, $G:=\left<g\right>$, $g:=\frac{1}{2}(0,1,1)$,  $\overline{f}:=y_1^2+y_2^2y_3+y_2y_3^{k+1}$.
%
%
%
%

\begin{proposition}{\label{EF4}}
    For the pair $(f,G)$ as above, the action of orbifold Gauss-Manin connection reads
\begin{enumerate}
\item If $\alpha+\beta \le 2k$, then
\[
\nabla_{\frac{\partial}{\partial v^\alpha_{\id}}}^{(f,G)}[x_3^{2\beta}\xi_{\id}] = \nabla^{F^G}_{\frac{\partial}{\partial v^\alpha_{\id}}}[x_3^{2\beta}]\xi_{\id}.
\]
If $\alpha+\beta = 2k+1$, then
\[
\nabla_{\frac{\partial}{\partial v^\alpha_{\id}}}^{(f,G)}[x_1^{2\beta}\xi_{\id}] = \nabla^{F^G}_{\frac{\partial}{\partial v^\alpha_{\id}}}[x_1^{2\beta}]\xi_{\id} - \frac{2v_g^0}{2k+1} \left[\frac{\p F^G}{\p v_\id^{\alpha+\beta-2k-1}}\right]
\xi_{g}.
\]
If $\alpha+\beta > 2k+1$, then
\[
\nabla_{\frac{\partial}{\partial v^\alpha_{\id}}}^{(f,G)}[x_1^{2\beta}\xi_{\id}] = \nabla^{F^G}_{\frac{\partial}{\partial v^\alpha_{\id}}}[x_1^{2\beta}]\xi_{\id} + \frac{(v_g^0)^2}{2k+1} 
\left[\frac{\p F^G}{\p v_\id^{\alpha+\beta-2k-2}}\right]\xi_{\id}.
\]

\item If $1 \le \beta \le 2k$, then we have
\[
\nabla_{\frac{\partial}{\partial v^0_{g}}}^{(f,G)}[x_3^{2\beta}\xi_{\id}] = \nabla_{\frac{\partial}{\partial v^\beta_{\id}}}^{(f,G)}[\xi_{g}] = - \frac{v_g^0}{2}\left[\frac{\partial F^G}{\partial v^{\beta-1}_{\id}}\right]\xi_{\id},
\]
and
\[
\nabla_{\frac{\partial}{\partial v^0_{g}}}^{(f,G)}[1 \cdot \xi_{\id}] 
= \nabla_{\frac{\partial}{\partial v^0_{\id}}}^{(f,G)}[\xi_{g}] 
= \xi_g,
\]
\item 
\[
        \nabla_{\frac{\partial}{\partial v^0_{g}}}^{(f,G)}[\xi_{g}] = 
        c_g \nabla_{\frac{\partial}{\partial v^0_{\id}}}^{(f,G)}[H^F_g \xi_{\id}](\widetilde v_{id})
    \]
where $\widetilde v_{\id}^{\alpha} := v_{\id}^{\alpha} \frac{2k + 1}{2\alpha -1}$.    

\end{enumerate}
\end{proposition}

\begin{proof}
 (i) In this case we can not apply the map $p$ and we prove the proposition by explicit calculations via the map $\Psi_{\bar{S}}$:
\[
 \nabla_{\frac{\partial}{\partial v^\alpha_{\id}}}^{(f,G)}[x_3^{2\beta}\xi_{\id}] = \Psi_{\bar{S}}[y_3^{\alpha+\beta}]
\]
If $0 \le \alpha+\beta \le 2k$, $\Psi_{\bar{S}} [y_3^{\alpha+\beta}]_{\bar{F}} = [x_3^{2(\alpha+\beta)}]\xi_{\id} = \nabla^{F^G}_{\frac{\partial}{\partial v^\alpha_\id}}[x_1^{2\beta}]\xi_{\id}$ and the statement holds. For $2k+1\le \alpha + \beta \le 4k$ we need to take into account the relations in the Brieskorn lattice $\H_{\bar F}$.
Denote $p := \alpha+\beta - (2k+1)$. Then $0 \le p \le 2k-1$. Applying $D_{\bar{F}}$ to $-[y_3^{p+k}dy_1\wedge dy_3]$ and $[y_3^{p+1}dy_1\wedge dy_2]$ we have
\[
[2y_2y_3^{p+k+1} + y_3^{\alpha+\beta} + \bar{v}_{0} y_3^{p+k}]_{\bar{F}} = [2y_3^{p+k}(y_2 y_3+ \frac{\bar{v}_0}{2}) + y_3^{\alpha+\beta}]_{\bar{F}}= 0,
\]
\[
[z(p+1)y_3^{p} + k y_3^{p+k}(y_2y_3 + \frac{\bar{v}_0}{2}) +  y_3^{p+1}(y_2^2 + y_2y_3^k) + \sum_{\gamma=0}^{2k}\gamma\bar{s}_\gamma y_3^{\gamma +p}]_{\bar{F}} = 0.
\]
These relations give us
\[
[y_3^{\alpha+\beta}]_{\bar{F}} = \frac{2}{k}[z(p+1)y_3^p + y_3^{p+1}(y_2^2 + y_2y_3^k) + \sum_{\gamma=0}^{2k}\gamma\bar{s}_\gamma y_3^{\gamma +p}]_{\bar{F}}.
\]
We also note that
\[
[y_3^{p+1}(y_2^2 + y_2y_3^k)]_{\bar{F}} = [y_3^{p+1}((y_2 + \frac{1}{2}y_3^k)^2 - \frac{1}{4}y_3^{2k})]_{\bar{F}} = [y_3^{p+1}(y_2 + \frac{1}{2}y_3^k)^2 - \frac{1}{4}y_3^{\alpha+\beta}]_{\bar{F}}
\]
which gives us
\[
[y_3^{\alpha+\beta}]_{\bar{F}} = \frac{4}{2k+1}[z(p+1)y_3^p + y_3^{p+1}(y_2 + \frac{1}{2}y_3^k)^2 + \sum_{\gamma=0}^{2k}\gamma\bar{s}_\gamma y_3^{\gamma +p}]_{\bar{F}}.
\]
Finally, applying $D_{\bar{F}}$ to $[y_3^{p}(y_2+\frac{1}{2}y_3^k) dy_1 \wedge dy_3]$ we have
\[
[z y_3^p + 2y_3^{p+1}(y_2+\frac{1}{2}y_3^k)^2 + \bar{v}_0 y_3^{p}(y_2+\frac{1}{2}y_3^k)]_{\bar{F}} = 0.
\]
This relation gives
\[
[y_3^{\alpha+\beta}]_{\bar{F}} = \frac{4}{2k+1}[z(p+\frac{1}{2})y_3^p + \sum_{\gamma=0}^{2k}\gamma\bar{s}_\gamma y_3^{\gamma +p}]_{\bar{F}} - \frac{4}{2k+1}[\frac{\bar{v}_0}{2}y_3^{p}(y_2+\frac{1}{2}y_3^k)]_{\bar{F}}.
\]
If $p > 0$ we need to express the last summand as well. Applying $D_{\bar F}$ to $[y_3^{p-1} dy_1 \wedge dy_3]$ we have
 \[
    [ 2y_3^p(y_2 + \frac{1}{2}y_3^k) + \bar v_0 y_3^{p-1}]_{\bar F} = 0,
 \]
 what gives the expression for the second term.

To complete the proof we show that the first term is determined by $\nabla^{F^G}_{\frac{\partial}{\partial v^\alpha_{\id}}}[x_3^{2\beta}]\xi_{\id}$. Note that $\nabla^{F^G}_{\frac{\partial}{\partial v^\alpha_{\id}}}[x_3^{2\beta}]\xi_{\id} = [x_3^{2(\alpha+\beta)}]\xi_{\id}$. If $0 \le \alpha + \beta \le 2k$, the statement is clear. If $2k+1 \le \alpha + \beta \le 4k$, applying $D_{F^G}$ to $[-x_3^{2p+2k+1} dx_1 \wedge d x_3]$ and $[x_3^{2p+1}dx_1 \wedge dx_2]$ we have
\[
[2x_2x_3^{2p+2k+1} + x_2^{2(\alpha+\beta)}]_{F^G} = 0,
\]
\[
[2z (p+\frac{1}{2})x_3^{2p} + (2k+1)x_2x_3^{2p+2k+1} + \sum_{\gamma=0}^{2k}2\gamma v^{\gamma}_{\id} x_3^{2(\gamma+p)}]_{F^G} = 0.
\]
These relations give us
\[
[x_3^{2(\alpha+\beta)}]\xi_{\id} = \frac{4}{2k+1}(z (p+\frac{1}{2})x_3^{2p} + \sum_{\gamma=0}^{2k}\gamma v^{\gamma}_{\id} x_3^{2(\gamma+p)}]\xi_{\id}.
\]
This completes the proof.

(ii) $\nabla_{\frac{\partial}{\partial v^0_{g}}}^{(f,G)}[x_1^{2\beta}\xi_{\id}] = \Psi_{\bar{S}}[y_3^{\beta}(y_2+\frac{1}{2}y_3^k)] _{\bar{F}}$. Applying $D_{\bar{F}}$ to $[-y_3^{\beta-1}dy_1\wedge dy_3]$ we obtain:

\[
[2y_3^\beta(y_2+\frac{1}{2}y_3^k)+ \bar{v}_0y_3^{\beta-1}] _{\bar{F}}= 0
\]
from where we calculate
\[
\nabla_{\frac{\partial}{\partial v^0_{g}}}^{(f,G)}[x_1^{2\beta}\xi_{\id}] = \Psi_{\bar{S}}[y_3^{\beta}(y_2+\frac{1}{2}y_3^k)] _{\bar{F}} = \Psi_{\bar{S}}[-\frac{\bar{v}_0}{2}y_3^{\beta-1}] _{\bar{F}}= - \frac{v_g^0}{2}[x_3^{2\beta-2}]\xi_{\id} =  - \frac{v_g^0}{2}\left[\frac{\partial F^G}{\partial v^{\beta-1}_{\id}}\right]\xi_{\id}.
\]
Moreover
\[ 
\nabla_{\frac{\partial}{\partial v^\alpha_{\id}}}^{(f,G)}[\xi_g] = \Psi_{\bar{S}}[y_3^{\beta}(y_2+\frac{1}{2}y_3^k)] _{\bar{F}} = 
\nabla_{\frac{\partial}{\partial v^0_{g}}}^{(f,G)}[x_3^{2\alpha}\xi_{\id}] 
\]
what proves (ii).

(iii) $\nabla_{\frac{\partial}{\partial v^0_{g}}}^{(f,G)}[\xi_{g}] = \Psi_{\bar{S}}[(y_2+\frac{1}{2}y_3^k)^2]_{\bar{F}} = \Psi_{\bar{S}}[y_2^2+ y_2y_3^k + \frac{y_3^{2k}}{4}]_{\bar{F}}$. 
Applying $D_{\bar{F}}$ to $[dy_1\wedge dy_2]$ we obtain:
\[
[y_2^2 + y_2y_3^k + ky_2y_3^k+ \sum_{\gamma = 0}^{2k} \gamma \bar{s}_{\gamma} y_3^{\gamma-1} + \bar v_0\frac{k}{2}y_3^{k-1}]_{\bar{F}} = 0.
\]
Applying $D_{\bar{F}}$ to $[-y_3^{k-1}dy_1\wedge dy_3]$ we obtain:
\[
    [2y_2y_3^k + y_3^{2k} + \bar v_0y_3^{k-1}]_{\bar{F}} = 0.
\]
This gives
\begin{align*}
 [(y_2+\frac{1}{2}y_3^k)^2]_{\bar{F}} &= [-ky_2y_3^k - \sum_{\gamma = 0}^{2k} \gamma \bar{s}_{\gamma} y_3^{\gamma-1} - \bar v_0\frac{k}{2}y_3^{k-1} + \frac{y_3^{2k}}{4}]_{\bar{F}}
 = [ \frac{2k+1}{4} y_3^{2k} - \sum_{\gamma = 0}^{2k} \gamma \bar{s}_{\gamma} y_3^{\gamma-1}]_{\bar{F}}, 
 \\
\nabla_{\frac{\partial}{\partial v^0_{g}}}^{(f,G)}[\xi_{g}] 
    &= \frac{1}{4}[(2k+1)x_3^{4k} - \sum_{\gamma=0}^{2k} 4\gamma v_{\gamma}^{\id} x_3^{2\gamma-2}]\xi_{\id}.
\end{align*}

Note that 
\begin{align*}
 c_gH^F_g = -\frac{H^F_g}{|G|^2} &= -\frac{1}{2k+1} \frac{1}{4} \left( \sum_{\gamma=0}^{2k} v^{\gamma}_{\id} 4\gamma(2\gamma-1)x_3^{2\gamma-1} - (2k+1)^2 x_3^{4k} \right)
 \\
 &=  \frac{1}{4} \left( (2k+1) x_3^{4k} - \sum_{\gamma=0}^{2k} v^{\gamma}_{\id} 4\gamma \frac{2\gamma-1}{2k+1}x_3^{2\gamma-1} \right)
\end{align*}
what completes the proof.

\end{proof}


\begin{thebibliography}{99}
\bibitem[AGV]{AGV} 
     V.I. ~Arnold, S.M. ~Gusein-Zade, A.N. ~Varchenko, \emph{Singularities of Differentiable Maps, Volume 1}, 1982.
\bibitem[AT22]{AT22}     
    L.~Amorim, J~Tu, \emph{Categorical primitive forms of Calabi–Yau $A_\infty$-categories with semi-simple cohomology}. Selecta Mathematica, 28(3), p.54. 2022
\bibitem[B25]{B25} 
     A. Basalaev, \emph{Mirror map for Fermat polynomials with a nonabelian group of symmetries}, Theoretical and Mathematical Physics, 209(2): 1491–1506 (2021).
\bibitem[BI21]{BI21} 
     A. Basalaev, A. Ionov, \emph{Hochschild cohomology of Fermat type polynomials with
non--abelian symmetries}, Journal of Geometry and Physics 174 (2022) 104450.
\bibitem[BI22]{BI22} 
     A. Basalaev, A. Ionov, \emph{Hochschild cohomology of Fermat type polynomials with
non--abelian symmetries}, Journal of Geometry and Physics 174 (2022) 104450.
\bibitem[BI24]{BI24} 
     A. Basalaev, A. Ionov, \emph{Hodge Diamonds of the Landau–Ginzburg Orbifolds}, Symmetry, Integrability and Geometry: Methods and Applications (SIGMA) 20 (2024), 024, 25 pages 
\bibitem[BT1]{BT1}
	A.~Basalaev, A.~Takahashi, \emph{Hochschild cohomology and orbifold Jacobian algebras associated to invertible polynomials}, Arnold Math J. (2017). https://doi.org/10.1007/s40598-017-0076-8.
\bibitem[BT2]{BT2}
	A.~Basalaev, A.~Takahashi, \emph{Mirror Symmetry for a Cusp Polynomial Landau–Ginzburg Orbifold}, International Mathematics Research Notices. 2022. Vol. 2022. No. 19. P. 14865-14922.
\bibitem[BTW1]{BTW1}
	A.~Basalaev, A.~Takahashi, E.~Werner, \emph{Orbifold Jacobian algebras for invertible polynomials}, arXiv preprint: 1608.08962.
\bibitem[BTW2]{BTW2} 
    A. Basalaev, A. Takahashi, E. Werner, \emph{Orbifold Jacobian algebras for exceptional unimodal singularities}, Arnold Math J. (2017).
%
\bibitem[CR11]{CR11}
    A. Chiodo, Y. Ruan.
    {\it A global mirror symmetry framework for the Landau-Ginzburg/Calabi-Yau correspondence},
    Annales de l'Institut Fourier. Vol. 61. No. 7. (2011) 2803--2864 .
\bibitem[EGZ]{EGZ} 
    W. Ebeling, S. M. Gusein-Zade, \emph{Saito duality between Burnside rings for invertible polynomials}, Bull. Lond. Math. Soc. 44 (2012), no.4, 814–822. 
\bibitem[ET1]{ET1}
	W.~Ebeling, A.~Takahashi, \emph{Variance of the exponents of orbifold Landau-Ginzburg models,}, Math. Res. Lett. 20 (2013), no.01, 51–65, arXiv: 1203.3947
 \bibitem[ET2]{ET2}
	W.~Ebeling, A.~Takahashi, \emph{A geometric definition of Gabrielov numbers.}, Rev. Mat. Complut. 27 (2014): 447–60.
\bibitem[HLL]{HLL}    
    W.~He, S.~Li, Y.~Li, \emph{Unfolding of orbifold LG B-models: a case study}. arXiv preprint arXiv:1904.09426. 2019
\bibitem[IV]{IV}
    K. A. Intriligator, C. Vafa, \emph{Landau-Ginzburg Orbifolds}, Nuclear Physics, B339:95–120 (1990).
\bibitem[Io]{Io} 
    A. Ionov, \emph{McKay correspondence and orbifold equivalence}, Journal of Pure and Applied Algebra Volume 227, (2023).
\bibitem[Orl1]{Orl1} 
    D. Orlov, \emph{Triangulated categories of singularities and D-branes in Landau- Ginzburg models}, Proc. Steklov Inst. Math. 2004, no. 3(246), 227–248.
\bibitem[Orl2]{Orl2} 
    D. Orlov, \emph{Matrix factorizations for nonaffine LG-models}, Math. Ann., 353, (2012), 95–108.
\bibitem[R1]{R1}
    A. Rarovskii, 
    \emph{Non-invertible quasihomogeneous singularities and their Landau-Ginzburg orbifolds}, arXiv preprint: 2405.17091, 2024.

\bibitem[SK81]{SK81} 
    K.~Saito, 
    {\it The higher residue pairings $K_F^{(k)}$ for a family of hypersurface singular points}. 
    Singularities, Part 2 (Arcata, Calif., 1981), 441–463, Proc. Sympos. Pure Math., 40, Amer. Math.Soc., Providence, RI, 1983.

\bibitem[SK83]{SK83} 
	K.~Saito
	{\it Period mapping associated to a primitive form}, Publ. RIMS 19 (1983) 1231 - 1264. 
	
\bibitem[SM89]{SM89}
    M.~Saito, {\it On the structure of Brieskorn lattice}. Ann. Inst. Fourier, vol. 39(1989), 27–72.

\bibitem[SaiT]{SaiT}
	K.~Saito and A.~Takahashi,
	{\it From Primitive Forms to Frobenius manifolds}, 
	Proc. Sympos. Pure Math., 78, 31--48, (2008).
\bibitem[Sh]{Sh}
	D.~Shklyarov, \emph{On Hochschild invariants of Landau--Ginzburg orbifolds}, arXiv preprint: 1708.06030v1.
\bibitem[Tu21a]{Tu21a} 
    J. Tu, \emph{Categorical Saito theory, I: A comparison result}. Advances in Mathematics, 383, p.107683. 2021
\bibitem[Tu21b]{Tu21b} 
    J. Tu, \emph{Categorical Saito Theory, II: Landau-Ginzburg orbifolds},  Advances in Mathematics, 384, p.107744. 2021.

 \bibitem[V]{V} 
    C. Vafa, \emph{String vacua and orbifoldized LG models}, Modern Physics Letters A 4.12 pp. 1169–1185 (1989).
\bibitem[Witt]{Witt} 
    E.Witten, \emph{Phases of N= 2 theories in two dimensions}, Nuclear Physics B. Aug 16;403(1-2):159–222 (1993).


\end{thebibliography}
\end{document}